\documentclass[11pt]{article}
   \usepackage[top=1in,bottom=1in,left=1.25in,right=1.25in]{geometry}
\usepackage{amssymb}
\usepackage{amsmath}
\usepackage{amsthm}
\usepackage{graphicx}
\usepackage{subfigure}
\usepackage{diagbox}
\usepackage{esint}
\usepackage{caption}
\usepackage{booktabs}
\usepackage{ulem}
\usepackage{epstopdf}
\usepackage{float}

\numberwithin{equation}{section}

\newtheorem{theorem}{Theorem}[section]
\newtheorem{lemma}{Lemma}[section]

\newcommand{\R}{\mathbb {R}}

\usepackage{color}

\date{}

\begin{document}

    \title{\bf A two-strain reaction-diffusion malaria model with seasonality and vector-bias\thanks{  This research is supported by the NSF of China (No. 11971369) and the Fundamental Research Funds for the Central Universities (No. JB210711).}}

\author{Huijie Chu,~~ Zhenguo Bai\\
    School of Mathematics and Statistics, Xidian
University, Xi'an, 710126, China\\
    E-mail:~~~zgbai@xidian.edu.cn}

\maketitle

\begin{abstract}
To investigate the combined effects of drug resistance, seasonality and vector-bias, we formulate a periodic two-strain reaction-diffusion model. It is a competitive system for
resistant and sensitive strains, but the single-strain subsystem is cooperative.
We derive the basic reproduction number $\mathcal {R}_i$ and the invasion reproduction number $\mathcal {\hat{R}}_i$ for strain $i~(i=1,2)$, and establish the transmission dynamics in terms of these four quantities. More precisely, (i) if $\mathcal {R}_1<1$ and $\mathcal{R}_2<1$, then the disease is extinct; (ii) if $\mathcal {R}_1>1>\mathcal{R}_2$ ($\mathcal {R}_2>1>\mathcal{R}_1$), then the sensitive (resistant) strains are persistent, while the resistant (sensitive) strains die out; (iii) if $\mathcal {R}_i>1$ and $\mathcal {\hat{R}}_i>1~(i=1,2)$, then two strains are coexistent and periodic oscillation phenomenon is observed. We also study the asymptotic behavior of the basic reproduction number with respect to small and large diffusion coefficients. Numerically, we demonstrate the phenomena of coexistence and competitive exclusion for two strains and explore the influences of seasonality and vector-bias on disease spreading.
\end{abstract}


\smallskip


\section{Introduction}
\label{Intro}
Malaria, one of the most common vector-borne diseases, is endemic in over 100 countries worldwide and causes serious public health problems and a significant economic burden worldwide \cite{JuanB2015}. Human malaria infection is caused by the genus Plasmodium parasite, which can be transmitted to humans by the effective bites of adult female Anopheles mosquitoes (after taking a blood meal from humans) \cite{Forouzannia and Gumel 2014}. According to the 2020 WHO report \cite{WHO}, the global tally of malaria cases was 229 million in 2019, claiming some 409 000 lives compared to 411 000 in 2018. Therefore, a deep understanding of malaria transmission mechanisms will undoubtedly contribute to disease control.

Mathematical models have been proposed to study the
dynamics of malaria outbreaks in different parts of the world, the earliest model dates back to the Ross-Macdonald model \cite{Ross 1911,Macdonald1957}. Since then,
various mathematical models have been designed to describe and predict the spreading of malaria (see, e.g., \cite{Cosner09, Chamchod and Britton 2011, Lou and Zhao 2010,LouY2011, Xiao and Zou 2014, Wang and Zhao 2017, Bai et al. 2018, WuZhao2019,Wangbg2020}). However, few studies consider the following three biological factors for malaria transmission simultaneously.

{\it Vector-bias effect}. The vector-bias describes that mosquitoes prefer biting infectious humans to susceptible ones. Kingsolver \cite{Kingsolver 1987} first introduced a vector-bias model for the dynamics of malarial transmission. Following Kingsolver's work, Hosack et al. \cite{Hosack2008} included the incubation time in mosquitoes to study the dynamics of the disease concerning the reproduction number. Further, Chamchod and Britton \cite{Chamchod and Britton 2011} extended the model from previous authors by defining the attractiveness in a different way. Motivated by these works, Wang and Zhao incorporated the seasonality into a vector-bias model with incubation period \cite{Wang and Zhao 2017}. Bai et al. formulated a time-delayed periodic reaction-diffusion model with vector-bias effect \cite{Bai et al. 2018} and found that the ignorance of the vector-bias effect will underestimate the infection risk.
All these results show that the vector-bias has an important impact on the epidemiology of malaria.

{\it Drug-resistance}. Currently, due to the lack of effective and safe vaccine, the main strategy in
controlling malaria is drugs. However, the use of anti-malarial drugs such as chloroquine, malaraquine, nivaquine, aralen and fansidar results in the appearance and spread of resistance in the parasite population \cite{Forouzannia and Gumel 2014,Aneke 2002,Klein 2013}. This poses a significant challenge to the global control of malaria transmission or eradication of the disease.
Therefore, it is essential to investigate the resistance in malaria transmission.

{\it Seasonality}. It is generally believed that climatic factors such as temperature, rainfall, humidity, wind, and duration of daylight greatly influence the transmission and distribution of vector-borne diseases \cite{Agusto et al. 2015,Cailly 2012,Hoshen and Morse 2004}. For example, rising temperatures will reduce the number of days required for breeding, and thereby increase mosquito development rates \cite{Ewing et al. 2016}. There have been some mathematical models and field observations suggesting that the strength and mechanisms of seasonality can change the pattern of infectious diseases \cite{Ewing et al. 2016,Altizer et al. 2010}. These results are beneficial for forecasting the mosquito abundance and further effectively controlling the disease.

Except these considerations above, human and vector populations have also contributed to the spread of vector-borne diseases \cite{Cosner09,LouY2011}. Therefore,
this paper will investigate a periodic two-strain malaria model with diffusion, which is an extension of autonomous limiting system in \cite{Shi and Zhao 2021}.
In view of the intrinsic mathematical structure of the model, we choose a time-varying phase space to carry out dynamical analysis. This idea has also been used in \cite{Li and Zhao 2021}.
In particular, we prove that no subset forms a cycle on the boundary with the aim of using uniform persistence theory. Its proof is nontrivial (see Theorem \ref{uniformly persistent}).

The rest of this paper is organized as follows. In the next section, we formulate the model and
study its well-posedness. In Section 3, we define the basic reproduction number $\mathcal {R}_i$ and the invasion reproduction number $\mathcal {\hat{R}}_i~(i=1,2)$ for the sensitive and resistant strains, respectively.
In Section 4, we investigate the uniform persistence and extinction in terms of the reproduction numbers.
In Section 5, we analyze the asymptotic behavior of the basic reproduction number concerning small and large diffusion coefficients. In Section 6, we conduct numerical
study for our model. And the paper ends with a brief discussion.

\section{Model formulation}

Motivated by \cite{Bai et al. 2018,Shi and Zhao 2021}, we consider the model with no immunity; that is, individuals who recovered from malaria cannot resist reinfection of the disease and can become susceptible directly. We assume that no susceptible individual or mosquito can be infected by two virus strains. The total human population $N_h(t,x)$ is divided into three groups: susceptible $S_h(t,x)$, infected individuals with drug sensitive strain $I_1(t,x)$ and infected individuals with drug resistant strain $I_2(t,x)$. For the vector population, only adult female mosquitoes can contract the virus due to adult males and immature mosquitoes do not take blood. Thereby, we consider only adult female mosquitoes in our model. The vector population $M(t,x)$ has the epidemiological classes denoted by $S_v(t,x)$, $I_{v1}(t,x)$ and $I_{v2}(t,x)$ for the susceptible, infected with sensitive and resistant strains, respectively.

Assume that all populations remain confined to a bounded domain $\Omega\subset \R^m (m\geq 1)$ with smooth boundary $\partial \Omega$ (when $m\geq 1$). Following the line in \cite{Bai et al. 2018},
we suppose that the density of total human population $N_h(t,x)=S_h(t,x)+I_1(t,x)+I_2(t,x)$ satisfies the following reaction-diffusion equation:
\begin{equation}
\label{total density of human}
\left\{
\begin{split}
\frac{\partial N_h(t,x)}{\partial t} & =D_h\Delta N_h(t,x)+B(x,N_h)N_h(t,x)-dN_h(t,x),~~~&& t>0,~x \in \Omega,\\
\frac{\partial N_h(t,x)}{\partial\nu} &=0,~~~&& t>0,~x \in \partial\Omega,
\end{split}
\right.
\end{equation}
with
\begin{equation*}
\begin{array}{l}  B(x,u)=\left\{\begin{array}{ll}
b\left[ 1-\frac{u}{K(x)}\right], & 0\leq u\leq K(x), ~x\in\bar{\Omega},\\
0,& u>K(x), ~x\in\bar{\Omega},
\end{array}\right.
\end{array}
\end{equation*}
where $\Delta$ is the usual Laplacian operator. $D_h>0$ is the diffusion coefficient of humans, $b$ and $d ~(0<d<b)$ are respectively the maximal birth rate and the nature mortality rate of humans,  and $K(x)$ denotes the local carrying capacity, which is supposed to be
a positive continuous function of location $x$. By employing \cite[Theorems 3.1.5 and 3.1.6]{Zhao 2017b}, we arrive at that system \eqref{total density of human}
admits a globally attractive positive steady state $N(x)$ in $C(\bar{\Omega},
\R_+)\setminus \{0\}$.

We also assume that the equation of the total mosquito population $M(t,x)=S_v(t,x)+I_{v1}(t,x)+I_{v2}(t,x)$ is of the form:
\begin{equation}
\label{total density of female adult mosquitoes}
\left\{
\begin{split}
\frac{\partial M(t,x)}{\partial t} &=D_v\Delta M(t,x)+\Lambda(t,x)-\eta(t,x)M(t,x),~~~&& t>0,~x \in \Omega,\\
\frac{\partial M(t,x)}{\partial \nu} &=0,~~~&& t>0,~x\in\partial\Omega,
\end{split}
\right.
\end{equation}
where $D_v>0$ is the diffusion coefficient of mosquitoes, $\Lambda(t,x)$ is the recruitment rate at which adult female mosquitoes emerge from larval at time $t$ and location $x$, and $\eta(t,x)$ is the natural death rate of mosquitoes at time $t$ and location $x$. Functions $\Lambda(t,x)$ and $\eta(t,x)$ are
H\"{o}lder continuous and nonnegative nontrivial on $\mathbb{R}\times\bar{\Omega}$, and $\omega$-periodic in $t$ for some $\omega>0$. It easily follows that system \eqref{total density of female adult mosquitoes} admits a globally stable positive $\omega$-periodic solution $M^*(t,x)$ in $C(\bar{\Omega},
\R_+)$ (see, e.g., \cite[Lemma 2.1]{Zhang et al. 2015}). Biologically, we may suppose that the total human and mosquito density at time $t$ and location $x$ respectively stabilize at $N(x)$ and $M^*(t,x)$, that is, $N_h(t,x)\equiv N(x)$ and $M(t,x)\equiv M^*(t,x)$ for all $t\geq 0$ and $x \in \Omega$.

For model parameters, since the impact of climate change on mosquitoes activities is much more than that on humans, the parameters corresponding to mosquitoes are assumed to be time-dependent.
To incorporate a vector-bias term into the model, we use the parameters $p$ and $l$ to describe the probabilities that a mosquito arrives at a human at random and picks the human if he is infectious and susceptible, respectively \cite{Chamchod and Britton 2011, Wang and Zhao 2017}. Since infectious humans are more attractive to mosquitoes, we assume $p\geq l>0$. Let $\beta(t,x)$ be the biting rate of mosquitoes at time $t$ and location $x$; $c_1$($\alpha_1$) be the transmission probability per bite from infectious mosquitoes (humans) with sensitive strain to susceptible humans (mosquitoes), and $c_2$($\alpha_2$) be the transmission probability per bite from infectious mosquitoes (humans) with resistant strain to susceptible humans (mosquitoes). According to the induction in \cite{Shi and Zhao 2021}, we obtain
\begin{align*}
J_1(I_1(t,x),I_2(t,x)):=&\frac{c_1\beta(t,x)l(N(x)-I_1(t,x)-I_2(t,x))}{p(I_1(t,x)+I_2(t,x))+l(N(x)-I_1(t,x)-I_2(t,x))},\\
J_2(I_1(t,x),I_2(t,x)):=&\frac{\alpha_1\beta(t,x)pI_1(t,x)}{p(I_1(t,x)+I_2(t,x))+l(N(x)-I_1(t,x)-I_2(t,x))},\\
J_3(I_1(t,x),I_2(t,x)):=&\frac{c_2\beta(t,x)l(N(x)-I_1(t,x)-I_2(t,x))}{p(I_1(t,x)+I_2(t,x))+l(N(x)-I_1(t,x)-I_2(t,x))},\\
J_4(I_1(t,x),I_2(t,x)):=&\frac{\alpha_2\beta(t,x)pI_2(t,x)}{p(I_1(t,x)+I_2(t,x))+l(N(x)-I_1(t,x)-I_2(t,x))},
\end{align*}
where $J_1(J_3)$ represents the number of newly infectious humans with sensitive (resistant) strain caused by an infected mosquito with sensitive (resistant) strain per unit time at time $t$ and location $x$; and $J_2$($J_4$) means the force of infection on mosquitoes
due to the contact with infectious humans with sensitive (resistant) strain.

Taking into account all of these assumptions, we obtain the following periodic reaction-diffusion model:
\begin{equation}
\label{model}
\left\{
\footnotesize
\begin{split}
&\frac{\partial I_1(t,x)}{\partial t}=D_h\Delta I_1(t,x)-(d+\gamma_1)I_1(t,x)+J_1(I_1(t,x),I_2(t,x))I_{v1}(t,x),~&&t>0,~x \in \Omega,\\
&\frac{\partial I_{v1}(t,x)}{\partial t}=D_v\Delta I_{v1}(t,x)-\eta(t,x)I_{v1}(t,x)\\
&~~~~~~~~~~~~~~~~~~~~~~~+J_2(I_1(t,x),I_2(t,x))(M^*(t,x)-I_{v1}(t,x)-I_{v2}(t,x)),~&&t>0,~x \in \Omega,\\
&\frac{\partial I_2(t,x)}{\partial t}=D_h\Delta I_2(t,x)-(d+\gamma_2)I_2(t,x)+J_3(I_1(t,x),I_2(t,x))I_{v2}(t,x),~&&t>0,~x \in \Omega,\\
&\frac{\partial I_{v2}(t,x)}{\partial t}=D_v\Delta I_{v2}(t,x)-\eta(t,x)I_{v2}(t,x)\\
&~~~~~~~~~~~~~~~~~~~~~~~+J_4(I_1(t,x),I_2(t,x))(M^*(t,x)-I_{v1}(t,x)-I_{v2}(t,x)),~&&t>0,~x \in \Omega,\\
&\frac{\partial I_1(t,x)}{\partial \nu}=\frac{\partial I_{v1}(t,x)}{\partial \nu}=\frac{\partial I_2(t,x)}{\partial \nu}=\frac{\partial I_{v2}(t,x)}{\partial \nu}=0,~&&t>0,~x\in\partial\Omega,\\
&I_1(0,x)=I_1^0(x)\geq 0,~~I_{v1}(0,x)=I_{v1}^0(x)\geq 0,\\
&I_2(0,x)=I_2^0(x)\geq 0,~~I_{v2}(0,x)=I_{v2}^0(x)\geq 0,~&&x \in\bar{\Omega}.
\end{split}
\right.
\end{equation}
Here, the positive constants $\gamma_1$ and $\gamma_2$ denote the recovery rate of the sensitive and resistant strains for humans, respectively. The function $\beta(t,x)$ is H\"{o}lder continuous and nonnegative but not zero identically on $\mathbb{R}\times\bar{\Omega}$, and $\omega$-periodic in $t$. Other parameters are the same as above.

Let $\mathbb{X}:=C(\bar{\Omega},\mathbb{R}^4)$ be the Banach space with supremum norm $\|\cdot \|$ and $\mathbb{X}^+:=C(\bar{\Omega},\mathbb{R}_+^4)$. For each $t\geq 0$, we define
\begin{equation*}
\begin{split}
X(t):=\{\varphi=(\varphi_1, \varphi_2, \varphi_3, \varphi_4)\in\mathbb{X}^+:~&0\leq\varphi_1(x)+\varphi_3(x)\leq N(x),\\
&0\leq\varphi_2(x)+\varphi_4(x)\leq M^*(t,x),~\forall x\in\bar{\Omega}\}.
\end{split}
\end{equation*}
Let $\mathbb{Y}:=C(\bar{\Omega},\mathbb{R})$ and $\mathbb{Y}^+:=C(\bar{\Omega},\mathbb{R}_+)$. Let
$T_1(t,s),~T_2(t,s),~T_3(t,s): \mathbb{Y}\rightarrow\mathbb{Y}, t\geq s$, be the linear evolution operators associated with
\begin{equation*}
\label{linear evolution operators}
\begin{split}
\frac{\partial v_1(t,x)}{\partial t}&=D_h\Delta v_1(t,x)-(d+\gamma_1)v_1(t,x):=A_1v_1(t,x),\\
\frac{\partial v_2(t,x)}{\partial t}&=D_h\Delta v_2(t,x)-(d+\gamma_2)v_2(t,x):=A_2v_2(t,x),\\
\frac{\partial v_3(t,x)}{\partial t}&=D_v\Delta v_3(t,x)-\eta(t,x)v_3(t,x):=A_3v_3(t,x),
\end{split}
\end{equation*}
subject to the Neumann boundary condition, respectively. Noting that $T_j(t,s)=T_j(t-s), j=1,2$, we have $T_j(t+\omega,s+\omega)=T_j(t,s)$ for $(t,s)\in \mathbb{R}^2$ with $t\geq s, j=1,2$. Since $\eta(t,x)$ is $\omega$-periodic in $t$, \cite[Lemma 6.1]{Daners and Koch 1992} implies that $T_3(t+\omega,s+\omega)=T_3(t,s)$ for $(t,s)\in \mathbb{R}^2$ with $t\geq s$. Moreover, for $(t,s)\in\mathbb{R}^2$ with $t>s$, $T_j(t,s), ~j=1,2,3$, are compact and strongly positive.
Set $T(t,s)={\rm diag}\{T_1(t,s), T_3(t,s), T_2(t,s), T_3(t,s)\}$ and $A(t)={\rm diag}\{A_1, A_3(t), A_2, A_3(t)\}$. Define $F=(F_1, F_2, F_3, F_4): [0,\infty)\times \mathbb{X}^+\rightarrow\mathbb{X}$ by
\begin{align*}
F_1(t,\varphi)&=\frac{c_1\beta(t,\cdot)l(N(\cdot)-\varphi_1(\cdot)-\varphi_3(\cdot))}{p(\varphi_1(\cdot)+\varphi_3(\cdot))+l(N(\cdot)-\varphi_1(\cdot)-\varphi_3(\cdot))}\varphi_2(\cdot),\\
F_2(t,\varphi)&=\frac{\alpha_1\beta(t,\cdot)p\varphi_1(\cdot)}{p(\varphi_1(\cdot)+\varphi_3(\cdot))+l(N(\cdot)-\varphi_1(\cdot)-\varphi_3(\cdot))}(M^*(t,\cdot)-\varphi_2(\cdot)-\varphi_4(\cdot)),\\
F_3(t,\varphi)&=\frac{c_2\beta(t,\cdot)l(N(\cdot)-\varphi_1(\cdot)-\varphi_3(\cdot))}{p(\varphi_1(\cdot)+\varphi_3(\cdot))+l(N(\cdot)-\varphi_1(\cdot)-\varphi_3(\cdot))}\varphi_4(\cdot),\\
F_4(t,\varphi)&=\frac{\alpha_2\beta(t,\cdot)p\varphi_3(\cdot)}{p(\varphi_1(\cdot)+\varphi_3(\cdot))+l(N(\cdot)-\varphi_1(\cdot)-\varphi_3(\cdot))}(M^*(t,\cdot)-\varphi_2(\cdot)-\varphi_4(\cdot))
\end{align*}
for all $t\geq 0$ and $\varphi=(\varphi_1, \varphi_2, \varphi_3, \varphi_4)\in \mathbb{X}$. Then system $\eqref{model}$
becomes
\begin{equation*}
\label{abstract equation}
\left\{
\begin{split}
&\frac{du}{dt}=A(t)u+F(t,u),~~~& t>0,\\
&u(0)=\varphi\in\mathbb{X}^+,
\end{split}
\right.
\end{equation*}
which can be written as an integral equation
\begin{equation}
\label{integral equation}
\begin{split}
u(t,\varphi)=T(t,0)\varphi+\int_{0}^{t}T(t,s)F(s,u)ds,~~~\forall t\geq 0,~\varphi\in \mathbb{X}^+,
\end{split}
\end{equation}
where
$$u(t,x)=(u_1(t,x),u_2(t,x),u_3(t,x),u_4(t,x))=(I_1(t,x),I_{v1}(t,x),I_2(t,x),I_{v2}(t,x)).$$
As usual, solutions of \eqref{integral equation} are called mild solutions to system \eqref{model}.\\


\begin{lemma}
\label{solution and periodic semiflow}
For any $\varphi\in X(0)$, system \eqref{model} has a unique solution $u(t,\cdot,\varphi)$ with $u(0,\cdot,\varphi)=\varphi$ such that $u(t,\cdot,\varphi)\in X(t)$ for all $t\in [0,\infty)$. Moreover, system \eqref{model} generates an $\omega$-periodic semiflow $Q(t): X(0)\rightarrow X(t)$ defined by $Q(t)\varphi=u(t,\cdot,\varphi), t\geq 0$. In addition, $Q:=Q(\omega)$ admits a strong global attractor in $X(0)$.
\end{lemma}

\begin{proof} From the expression of $F$, we see that $F$ is locally Lipschitz continuous. For any $(t,\psi)\in \R_+\times \mathbb{X}^+$ and $h>0$, in view of $p\geq l>0$, we have
\begin{equation*}
\begin{array}{c}
\begin{aligned}
&\psi(x)+h F(t,\psi)(x)\\
&=\left(
\begin{aligned}
& \psi_1(x)+h\frac{c_1\beta(t,x)l(N(x)-\psi_1(x)-\psi_3(x))}{p(\psi_1(x)+\psi_3(x))+l(N(x)-\psi_1(x)-\psi_3(x))}\psi_2(x)\\
&\psi_2(x)+h\frac{ \alpha_1\beta(t,x)p\psi_1(x) (M^*(t,x)-\psi_2(x)-\psi_4(x))}{ p(\psi_1(x)+\psi_3(x))+l(N(x)-\psi_1(x)-\psi_3(x))}\\
& \psi_3(x)+h \frac{c_2\beta(t,x)l(N(x)-\psi_1(x)-\psi_3(x))}{p(\psi_1(x)+\psi_3(x))+l(N(x)-\psi_1(x)-\psi_3(x))}\psi_4(x)\\
& \psi_4(x)+h\frac{ \alpha_2\beta(t,x)p\psi_3(x)(M^*(t,x)-\psi_2(x)-\psi_4(x)) }{ p(\psi_1(x)+\psi_3(x))+l(N(x)-\psi_1(x)-\psi_3(x))}
\end{aligned}
\right)\\
&\geq\left(
\begin{aligned}
&\psi_1(x)\left(1-h\frac{c_1\beta(t,x)}{N(x)}\psi_2(x)\right)-h\frac{c_1\beta(t,x)}{N(x)}\psi_3(x)\psi_2(x)\\
& \psi_2(x)\left(1-h\frac{\alpha_1\beta(t,x)p}{lN(x)}\psi_1(x)\right)-h\frac{\alpha_1\beta(t,x)p}{lN(x)}\psi_1(x)\psi_4(x)\\
&
\psi_3(x)\left(1-h\frac{c_2\beta(t,x)}{N(x)}\psi_4(x)\right)-h\frac{c_2\beta(t,x)}{N(x)}\psi_1(x)\psi_4(x)\\
& \psi_4(x)\left(1-h\frac{\alpha_2\beta(t,x)p}{lN(x)}\psi_3(x)\right)-h\frac{\alpha_2\beta(t,x)p}{lN(x)}\psi_3(x)\psi_2(x)
\end{aligned}
\right).
\end{aligned}
\end{array}
\end{equation*}
This implies that
$$\lim_{h\rightarrow 0^+}\frac{1}{h}{\rm dist}(\psi+hF(t,\psi),~\mathbb{X}^+)=0,~~~\forall (t,\psi)\in \R_+\times\mathbb{X}^+.$$
In addition, $T(t,s)\mathbb{X}^+\subseteq\mathbb{X}^+, \forall t\geq s\geq 0$. Therefore, by
\cite[Corollary 4]{Martin and Smith 1990} with $K=\mathbb{X}^+$ and $S(t,s)=T(t,s)$, system \eqref{model} admits a unique non-continuable mild solution $u(t,\cdot,\psi)$ on its maximal existence interval $[0,t_\psi)$ with $u(0,\cdot,\psi)=\psi$, and $u(t,\cdot,\psi)\geq 0$ for all $t\in[0,t_\psi)$, where $t_{\psi}\leq\infty$.

Based on the above analysis, we obtain that for any $\varphi\in X(0)\subset \mathbb{X}^+$, system \eqref{model} has a unique
solution $u(t,\cdot,\varphi)\in \mathbb{X}^+$ on $[0,t_\varphi)$ with $u(0,\cdot,\varphi)=\varphi$, where $t_{\varphi}\leq\infty$.  Next we want to show that $u(t,x,\varphi)$ is bounded for all $t\in [0,t_\varphi)$, which then
implies $t_\varphi=\infty$. To this end, we set
\[I_h(t,x)=I_1(t,x)+I_2(t,x), ~~~I_v(t,x)=I_{v1}(t,x)+I_{v2}(t,x).\]
It turns out that $N(x)$ and $M^*(t,x)$ are respectively the upper solutions of the following two equations
\begin{align*}
\frac{\partial I_h(t,x)}{\partial t}=&D_h\Delta I_h(t,x)-dI_h(t,x)-\gamma_1 I_1(t,x)-\gamma_2 I_2(t,x)\\
& +J_1(I_1(t,x), I_2(t,x))I_{v1}(t,x)+J_3(I_1(t,x), I_2(t,x))I_{v2}(t,x),
\end{align*}
and
\begin{align*}
\frac{\partial I_v(t,x)}{\partial t}=&D_h\Delta I_v(t,x)-\eta(t,x)I_v(t,x)\\
&+J_2(I_1(t,x), I_2(t,x))(M^*(t,x)-I_v(t,x))\\
&+J_4(I_1(t,x), I_2(t,x))(M^*(t,x)-I_v(t,x))
\end{align*}
for $t\in (0,t_\varphi)$ and $x\in \bar{\Omega}$. Thus, the comparison principle implies that solutions of \eqref{model} are bounded on $[0, t_\varphi)$, and thus, $t_\varphi=\infty.$ In addition, we also have that $u(t,\cdot,\varphi)\in X(t)$ for all $t\geq 0$, and it is classic for $t>0$ in light of the analyticity of $T(t,s)$ when $t>s.$

Define a family of operators $\{Q(t)\}_{t\geq 0}$ from $X(0)$ to $X(t)$ by \[[Q(t)\varphi](x)=u(t,x,\varphi),~~~\forall\varphi\in X(0),~x\in\bar{\Omega}.\]
 By the proof of \cite[Lemma 2.1]{Zhang et al. 2015}, we can show that $Q(t)$ is an $\omega$-periodic semiflow, and thus $Q:=Q(\omega): X(0)\rightarrow X(\omega)=X(0)$ is the Poincar\'{e} map associated with system \eqref{model}. The fact that $u(t,\cdot,\varphi)\in X(t)$ for all $t\geq 0$ when $\varphi\in X(0)$ also implies that solutions of \eqref{model} are ultimately bounded. Hence, by \cite[Theorem 2.9]{Magal and Zhao 2005}, $Q$ has a strong global attractor in $X(0)$.
\end{proof}

\begin{lemma}
\label{solution}
For any $\varphi\in X(0)$, let $u(t,x,\varphi)$ be the solution of system \eqref{model}. If there exists some $t_0\geq0$ such that $u_i(t_0,x,\varphi)\not\equiv0,~i=1,2,3,4$, then
$$u_i(t,x,\varphi)>0,~~~i=1,2,3,4,~\forall t\geq t_0,~x\in\bar{\Omega}.$$
\end{lemma}

\begin{proof}
For any given $\varphi\in X(0)$, one easily sees
\begin{equation*}
\left\{
\begin{split}
\frac{\partial u_1(t,x)}{\partial t}&\geq D_h\Delta u_1(t,x)-(d+\gamma_1)u_1(t,x),~~~&&t>0,~x \in \Omega,\\
\frac{\partial u_2(t,x)}{\partial t}&\geq D_v\Delta u_2(t,x)-\bar{\eta}u_2(t,x),~~~&&t>0,~x \in \Omega,\\
\frac{\partial u_3(t,x)}{\partial t}&\geq D_h\Delta u_3(t,x)-(d+\gamma_2)u_3(t,x),~~~&&t>0,~x \in \Omega,\\
\frac{\partial u_4(t,x)}{\partial t}&\geq D_v\Delta u_4(t,x)-\bar{\eta}u_4(t,x),~~~&&t>0,~x \in \Omega,\\
\frac{\partial u_1(t,x)}{\partial \nu}&=\frac{\partial u_2(t,x)}{\partial \nu}=\frac{\partial u_3(t,x)}{\partial \nu}=\frac{\partial u_4(t,x)}{\partial \nu}=0,~~~&&t>0,~x\in\partial\Omega,
\end{split}
\right.
\end{equation*}
where $\bar{\eta}=\max_{(t,x)\in[0,\omega]\times\bar{\Omega}}\eta(t,x)$. If there exists $t_0\geq 0$ such that $u_i(t_0,x,\varphi)\not\equiv0$ for some $i\in\{1,2,3,4\}$, it then follows from the parabolic maximum principle \cite[Proposition 13.1]{P.Hess} that $u_i(t,x,\varphi)>0$ for all $t>t_0$ and $x\in\bar{\Omega}$.
\end{proof}

\section{Reproduction numbers}

In this section, we first define the basic reproduction number $\mathcal{R}_0$ of \eqref{model}, and then introduce the invasion reproduction number $\hat{\mathcal{R}}_i$ for strain $i~(i=1,2)$.

\subsection{Basic reproduction number}

In order to derive the basic reproduction number of \eqref{model}, we first consider subsystems: one involves sensitive strains alone and the other involves resistant strains alone. We fix $i\in\{1,2\}$ and let $I_j(t,x)\equiv 0,I_{vj}(t,x)\equiv 0,~\forall t\geq 0,~x\in\bar{\Omega},~j=1,2$ and $j\neq i$. Then system \eqref{model} reduces to the following single-strain model:
\begin{equation}
\label{single-strain model}
\left\{
\begin{split}
\frac{\partial I_i(t,x)}{\partial t}=&D_h\Delta I_i(t,x)-(d+\gamma_i)I_i(t,x)\\
&+\frac{c_i\beta(t,x)l(N(x)-I_i(t,x))}{pI_i(t,x)+l(N(x)-I_i(t,x))}I_{vi}(t,x),~&&t>0,~x \in \Omega,\\
\frac{\partial I_{vi}(t,x)}{\partial t}=&D_v\Delta I_{vi}(t,x)-\eta(t,x)I_{vi}(t,x)\\
&+\frac{\alpha_i\beta(t,x)pI_i(t,x)}{pI_i(t,x)+l(N(x)-I_i(t,x))}(M^*(t,x)-I_{vi}(t,x)),~&&t>0,~x \in \Omega,\\
\frac{\partial I_i(t,x)}{\partial \nu}=&\frac{\partial I_{vi}(t,x)}{\partial \nu}=0,~&&t>0,~x\in\partial\Omega,
\end{split}
\right.
\end{equation}

Let $\mathbb{E}:=C(\bar{\Omega},\mathbb{R}^2)$ and $\mathbb{E}^+:=C(\bar{\Omega},\mathbb{R}_+^2)$.
Linearizing \eqref{single-strain model} at $(0, 0)$ yields
\begin{equation}
\label{linearizing the single-strain model}
\left\{
\begin{split}
\frac{\partial I_i(t,x)}{\partial t}=&D_h\Delta I_i(t,x)-(d+\gamma_i)I_i(t,x)\\
&+c_i\beta(t,x)I_{vi}(t,x),~&&t>0,~x \in \Omega,\\
\frac{\partial I_{vi}(t,x)}{\partial t}=&D_v\Delta I_{vi}(t,x)-\eta(t,x)I_{vi}(t,x)\\
&+\frac{\alpha_i\beta(t,x)pM^*(t,x)}{lN(x)}I_i(t,x),~&&t>0,~x \in \Omega,\\
\frac{\partial I_i(t,x)}{\partial \nu}=&\frac{\partial I_{vi}(t,x)}{\partial \nu}=0,~&&t>0,~x\in\partial\Omega.\\
\end{split}
\right.
\end{equation}
Define the operator $\mathcal{F}_i(t):\mathbb{E}\rightarrow\mathbb{E}$ by
$$
\mathcal{F}_i(t)
\left(
\begin{array}{c}
\psi_{1}\\
\psi_{2}
\end{array}
\right)
=
\left(
\begin{array}{c}
c_i\beta(t,\cdot)\psi_2(\cdot)\\
\frac{\alpha_i\beta(t,\cdot)pM^*(t,\cdot)}{lN(\cdot)}\psi_1(\cdot)
\end{array}
\right),~~~\forall t\in\mathbb{R},~\psi=(\psi_1, \psi_2)\in\mathbb{E}.
$$
Let $-V_i(t)v=D\Delta v-W_i(t)v$, where $D={\rm diag}(D_h,D_v)$ and
$$
-[W_i(t)](x)
=
\left(
\begin{array}{cc}
-(d+\gamma_i)&0\\
0& -\eta(t,x)
\end{array}
\right),~~~\forall t\in\mathbb{R},~x\in\bar{\Omega}.
$$
Then $\Psi_i(t,s)={\rm diag}(T_i(t,s),T_3(t,s)),~t\geq s$, is the evolution operator on $\mathbb{E}$ associated with the following system
$$\frac{dv}{dt}=-V_i(t)v$$
subject to the Neumann boundary condition.
The exponential growth bound of $\Psi_i(t,s)$ is defined as
$$\bar{\omega}(\Psi_i)=\inf\{\tilde{\omega}_i:~\exists M\geq 1~\mbox{such that}~\|\Psi_i(t+s,s)\|_{\mathbb{E}}\leq Me^{\tilde{\omega}_it},~\forall s\in\mathbb{R},~t\geq0\}.$$
By the Krein-Rutman Theorem and \cite[Lemma 14.2]{P.Hess}, we have
$$0<r(\Psi_i(\omega,0))=\max\{r(T_i(\omega,0)), r(T_3(\omega,0))\}<1,$$
where $r(\Psi_i(\omega,0))$ is the spectral radius of $\Psi_i(\omega,0)$. Then, it follows from \cite[Proposition 5.5]{Thieme 2009} with $s=0$ that $\bar{\omega}(\Psi_i)<0$. Note that $\Psi_i(t,s)$ is a positive operator in the sense that $\Psi_i(t,s)\mathbb{E}^+\subseteq \mathbb{E}^+$ for all $t\geq s$. Therefore, $\mathcal{F}_i(t)$ and $\Psi_i(t,s)$ satisfy
\begin{itemize}
\item [(H1)] For each $t\geq 0$, $\mathcal{F}_i(t)$ is a positive operator on $\mathbb{E}$.
\item [(H2)] For any $t\geq s$, $\Psi_i(t,s)$ is a positive operator on $\mathbb{E}$, and $\bar{\omega}(\Psi_i)<0$.
\end{itemize}

Let $C_\omega(\mathbb{R},\mathbb{E})$ be the Banach space of all $\omega$-periodic and continuous functions from $\mathbb{R}$ to $\mathbb{E}$ equipped with the maximum norm. Following the theory developed in \cite{Zhao 2017a,Liang et al. 2019}, we define two linear operators on $C_\omega(\mathbb{R},\mathbb{E})$ by
$$[\mathcal{L}_iv](t):=\int^{\infty}_{0}\Psi_i(t,t-s)\mathcal{F}_i(t-s)v(t-s)ds,~~~\forall t\in\mathbb{R},~v\in C_\omega(\mathbb{R},\mathbb{E}),~i=1,2.$$
Motivated by the concept of next generation operators \cite{Thieme 2009,Bacaer and Guernaoui 2006}, we define the basic reproduction number as $\mathcal{R}_i:=r(\mathcal{L}_i)$, where $r(\mathcal{L}_i)$ is the spectral radius of $\mathcal{L}_i$.

The disease-free state of \eqref{model} is $(0, 0, 0, 0)$ and the corresponding linearized system is
\begin{equation}
\label{linearizing the model}
\left\{
\footnotesize
\begin{split}
&\frac{\partial I_1(t,x)}{\partial t}=D_h\Delta I_1(t,x)-(d+\gamma_1)I_1(t,x)+c_1\beta(t,x)I_{v1}(t,x),~~~&&t>0,~x \in \Omega,\\
&\frac{\partial I_{v1}(t,x)}{\partial t}=D_v\Delta I_{v1}(t,x)-\eta(t,x)I_{v1}(t,x)+\frac{\alpha_1\beta(t,x)pM^*(t,x)}{lN(x)}I_1(t,x),~~~&&t>0,~x \in \Omega,\\
&\frac{\partial I_2(t,x)}{\partial t}=D_h\Delta I_2(t,x)-(d+\gamma_2)I_2(t,x)+c_2\beta(t,x)I_{v2}(t,x),~~~&&t>0,~x \in \Omega,\\
&\frac{\partial I_{v2}(t,x)}{\partial t}=D_v\Delta I_{v2}(t,x)-\eta(t,x)I_{v2}(t,x)+\frac{\alpha_2\beta(t,x)pM^*(t,x)}{lN(x)}I_2(t,x),~~~&&t>0,~x \in \Omega,\\
&\frac{\partial I_1(t,x)}{\partial \nu}=\frac{\partial I_{v1}(t,x)}{\partial \nu}=\frac{\partial I_2(t,x)}{\partial \nu}=\frac{\partial I_{v2}(t,x)}{\partial \nu}=0,~~~&&t>0,~x\in\partial\Omega.\\
\end{split}
\right.
\end{equation}
Similarly, we can derive the basic reproduction number of \eqref{model}, which is given by \[\mathcal{R}_0=\max\{\mathcal{R}_1,\mathcal{R}_2\}.\]
For any given $t\geq 0$, let $P_i(t)$ be the solution map of \eqref{linearizing the single-strain model} on $\mathbb{E}$. Then $P_i:=P_i(\omega)$ is the associated Poincar\'{e} map. Let $r(P_i)$ be the spectral radius of $P_i$. By \cite[Theorem 3.7]{Liang et al. 2019} with $\tau=0$, we have the following nice property.

\begin{lemma}
\label{same sign R_i}
$\mathcal{R}_i-1$ has the same sign as $r(P_i)-1, i=1,2$, and thus $\mathcal{R}_0-1$ has the same sign as $r(P)-1$, where $r(P)=\max\{r(P_1),r(P_2)\}$ is the spectral radius of the Poincar\'{e} map $P$ associated with \eqref{linearizing the model}.
\end{lemma}


\subsection{Invasion reproduction number}
\label{Section 3.2}

In this subsection, we define the invasion reproduction number for each strain. The invasion reproduction number gives the ability of strain $i~(i=1,2)$ to invade strain $j~(j=1,2,~j\neq i)$ measured as the number of secondary infections strain $i$ one-infected individual can produce in a population where strain $j$ is at an endemic state \cite{Tuncer and Martcheva 2012}. We express it by $\hat{\mathcal{R}_i}~(i=1,2)$ and give their definition by analyzing the boundary $\omega$-periodic solution of \eqref{model}, that is, the sensitive strain $\omega$-periodic solution or resistant strain $\omega$-periodic solution.

For each $t\geq 0$, let $E(t)$ be subset in $\mathbb{E}$ defined by
$$E(t):=\{\psi=(\psi_1,\psi_2)\in\mathbb{E}^+: 0\leq\psi_1(x)\leq N(x),~ 0\leq\psi_2(x)\leq M^*(t,x),~\forall x\in\bar{\Omega}\}.$$
After a similar process in \cite[Lemma 3]{Li and Zhao 2021}, we obtain that for any $\psi\in E(0)$, system \eqref{single-strain model} has a unique solution $v_i(t,\cdot,\psi)=(I_i(t,x),I_{vi}(t,x))$ with $v_i(0,\cdot,\psi)=\psi$ such that $v_i(t,\cdot,\psi)\in E(t)$ for all $t\geq 0$. Moreover, by employing the arguments in \cite[Theorem 1]{Li and Zhao 2021}, one immediately obtains the following result.
\begin{theorem}
\label{single-strain solution}
The following statements are valid:
\begin{itemize}
\item [(i)] If $\mathcal{R}_i\leq 1$, then $(0, 0)$ is globally asymptotically stable for \eqref{single-strain model} in $E(0)$;
\item [(ii)] If $\mathcal{R}_i>1$, then system \eqref{single-strain model} admits a unique positive $\omega$-periodic solution $(I_i^*(t,x),I_{vi}^*(t,x))$, and it is globally asymptotically stable for \eqref{single-strain model} in $E(0)\setminus\{(0,0)\}$.
\end{itemize}
\end{theorem}
For ease of presentation, we introduce the following notations:\\
$\bullet$ $E_0=(0, 0, 0, 0)$: The disease-free state of \eqref{model}.\\
$\bullet$ $E_1(t,x)=(I_1^*(t,x), I_{v1}^*(t,x), 0, 0)$: The sensitive strain $\omega$-periodic solution of \eqref{model}.\\
$\bullet$  $E_2(t,x)=(0, 0, I_2^*(t,x), I_{v2}^*(t,x))$: The resistant strain $\omega$-periodic solution of \eqref{model}.

By Theorem 3.1, we see that when $\mathcal{R}_i>1~(i=1,2)$, system \eqref{model} admits a unique semitrivial boundary $\omega$-periodic solution $E_i(t,x)$. Linearizing \eqref{model} at the $E_j(t,x),~j\neq i,~i,j=1,2$, and considering only the equations for $I_i(t,x)$ and $I_{vi}(t,x)$, we get
\begin{equation}
\label{linearizing boundary}
\left\{
\begin{split}
\frac{\partial I_i(t,x)}{\partial t}=&D_h\Delta I_i(t,x)-(d+\gamma_i)I_i(t,x)\\
&+\frac{c_i\beta(t,x)l(N(x)-I_j^*(t,x))}{pI_j^*(t,x)+l(N(x)-I_j^*(t,x))}I_{vi}(t,x),~~~&&t>0,~x \in \Omega,\\
\frac{\partial I_{vi}(t,x)}{\partial t}=&D_v\Delta I_{vi}(t,x)-\eta(t,x)I_{vi}(t,x)\\
&+\frac{\alpha_i\beta(t,x)p(M^*(t,x)-I_{vj}^*(t,x))}{pI_j^*(t,x)+l(N(x)-I_j^*(t,x))}I_i(t,x),~~~&&t>0,~x \in \Omega,\\
\frac{\partial I_i(t,x)}{\partial \nu}=&\frac{\partial I_{vi}(t,x)}{\partial \nu}=0,~~~&&t>0,~x\in\partial\Omega.\\
\end{split}
\right.
\end{equation}
Similar to Section 3.1, we can define the invasion reproduction numbers $\hat{\mathcal{R}}_i~(i=1,2)$. Further, we have the following characterization of $\hat{\mathcal{R}}_i$.

\begin{lemma}
\label{same sign i}
$\hat{\mathcal{R}}_i-1$ has the same sign as $r(\hat{P}_i)-1$, where $\hat{P}_i$ is the Poincar\'{e} map associated with \eqref{linearizing boundary}, and $r(\hat{P}_i)$ is the spectral radius of $\hat{P}_i$.\
\end{lemma}

\section{Disease extinction and uniform persistence}

In this section, we establish the dynamics of \eqref{model} in terms of $\mathcal{R}_i$ and $\hat{\mathcal{R}}_i, i=1,2$.

\subsection{Global extinction}

\begin{theorem}
\label{global attractive}
If $\mathcal{R}_1<1$ and $\mathcal{R}_2<1$, then $E_0$ is globally attractive for \eqref{model} in $X(0)$.
\end{theorem}

\begin{proof} Let $(I_1(t,x),I_{v1}(t,x),I_2(t,x),I_{v2}(t,x))$ be the solution of \eqref{model} with initial data $\varphi\in X(0)$.  It is easily seen that
\begin{equation}
\label{compare model-a}
\left\{
\begin{split}
\frac{\partial I_1(t,x)}{\partial t}\leq & D_h\Delta I_1(t,x)-(d+\gamma_1)I_1(t,x)\\
&+\frac{c_1\beta(t,x)l(N(x)-I_1(t,x))}{pI_1(t,x)+l(N(x)-I_1(t,x))}I_{v1}(t,x),~&&t>0,~x \in \Omega,\\
\frac{\partial I_{v1}(t,x)}{\partial t}\leq & D_v\Delta I_{v1}(t,x)-\eta(t,x)I_{v1}(t,x)\\
&+\frac{\alpha_1\beta(t,x)pI_1(t,x)}{pI_1(t,x)+l(N(x)-I_1(t,x))}(M^*(t,x)-I_{v1}(t,x)),~&&t>0,~x \in \Omega,\\
\frac{\partial I_1(t,x)}{\partial \nu}=& \frac{\partial I_{v1}(t,x)}{\partial \nu}=0,~&&t>0,~x\in\partial\Omega.
\end{split}
\right.
\end{equation}
and
\begin{equation}
\label{compare model-b}
\left\{
\begin{split}
\frac{\partial I_2(t,x)}{\partial t}\leq & D_h\Delta I_2(t,x)-(d+\gamma_2)I_2(t,x)\\
&+\frac{c_2\beta(t,x)l(N(x)-I_2(t,x))}{pI_2(t,x)+l(N(x)-I_2(t,x))}I_{v2}(t,x),~&&t>0,~x \in \Omega,\\
\frac{\partial I_{v2}(t,x)}{\partial t}\leq & D_v\Delta I_{v2}(t,x)-\eta(t,x)I_{v2}(t,x)\\
&+\frac{\alpha_2\beta(t,x)pI_2(t,x)}{pI_2(t,x)+l(N(x)-I_2(t,x))}(M^*(t,x)-I_{v2}(t,x)),~&&t>0,~x \in \Omega,\\
\frac{\partial I_2(t,x)}{\partial \nu}=&\frac{\partial I_{v2}(t,x)}{\partial \nu}=0,~&&t>0,~x\in\partial\Omega.
\end{split}
\right.
\end{equation}
When $\mathcal{R}_1<1$, Theorem 3.1 implies that $(0, 0)$ is globally stable for \eqref{single-strain model}. Hence the comparison principle applies to \eqref{compare model-a} and ensures that
$$\lim_{t\rightarrow\infty}(I_1(t,x),I_{v1}(t,x))=(0,0)~\mbox{uniformly for}~x\in\bar{\Omega}.$$
In the case where $\mathcal{R}_2<1$, by using the similar procedure as above to \eqref{compare model-b}, one attains
$$\lim_{t\rightarrow\infty}(I_2(t,x),I_{v2}(t,x))=(0,0)~\mbox{uniformly for}~x\in\bar{\Omega}.$$
Therefore, the desired result is established.
\end{proof}

\subsection{Competitive exclusion and coexistence}

\begin{theorem}
\label{Competitive exclusion} Let $(I_1(t,\cdot,\varphi),I_{v1}(t,\cdot,\varphi),I_2(t,\cdot,\varphi),I_{v2}(t,\cdot,\varphi))$
be the solution of \eqref{model} through $\varphi\in X(0)$. Then the following assertions hold.
\begin{itemize}
  \item [(1)] If $\mathcal {R}_1>1>\mathcal {R}_2$ and $\varphi_1, \varphi_2\not\equiv 0$, then
  \[\lim_{t\rightarrow \infty} ((I_1(t,x,\varphi),I_{v1}(t,x,\varphi),I_2(t,x,\varphi),I_{v2}(t,x,\varphi))-E_1(t,x))=(0,0,0,0)\]
  uniformly for $x\in \bar{\Omega}$.
  \item [(2)] If $\mathcal {R}_2>1>\mathcal {R}_1$ and $\varphi_3, \varphi_4\not\equiv 0$, then
 \[\lim_{t\rightarrow \infty} ((I_1(t,x,\varphi),I_{v1}(t,x,\varphi),I_2(t,x,\varphi),I_{v2}(t,x,\varphi))-E_2(t,x))=(0,0,0,0)\]
  uniformly for $x\in \bar{\Omega}$.
\end{itemize}
\end{theorem}

\begin{proof} We only prove statement (1), since statement (2) can be treated similarly. In the case where $\mathcal {R}_2<1$, one immediately has that $\lim_{t\rightarrow \infty} (I_2(t,\cdot,\varphi),I_{v2}(t,\cdot,\varphi))=(0,0)$.
Then the limiting system of \eqref{model} is the system \eqref{single-strain model} with $i=1$.
Moreover, by employing the theory of internally chain transitive sets (see, e.g., \cite{Zhao 2017b}), we conclude that \[\lim_{t\rightarrow \infty} (I_1(t,x,\varphi),I_{12}(t,x,\varphi)-(I_1^*(t,x),I_{v1}^*(t,x))=(0,0)\]
uniformly for $x\in\bar{\Omega}$. Hence, statement (1) is established.
\end{proof}

For each $t\geq 0$, define
\[X_0(t):=\{(\varphi_1,\varphi_2,\varphi_3,\varphi_4)\in\ X(t): ~\varphi_i(\cdot)\not\equiv0,~i=1,2,3,4\},\]
and
\begin{align*}
\partial X_0(t):=X(t)\backslash X_0(t)=\{(\varphi_1,\varphi_2,\varphi_3,\varphi_4)\in\ X(t):~&\varphi_i(\cdot)\equiv0~\mbox{at least for one}~i\}.
\end{align*}
In order to study the coexistence of strains, we first give the following lemma for our subsequent coexistence result.

\begin{lemma}
\label{weak repeller}
Let $(I_1(t,\cdot,\varphi),I_{v1}(t,\cdot,\varphi),I_2(t,\cdot,\varphi),I_{v2}(t,\cdot,\varphi))$ be the solution of \eqref{model} with the initial value $\varphi\in X_0(0)$. If $\mathcal{R}_1>1$ and $\mathcal{R}_2>1$, then there exists $\delta>0$ such that
$$\limsup_{t\rightarrow\infty}\|(I_1(t,\cdot,\varphi),I_{v1}(t,\cdot,\varphi),I_2(t,\cdot,\varphi),I_{v2}(t,\cdot,\varphi))-E_0\|\geq\delta.$$
\end{lemma}

\begin{proof}
Suppose, by contradiction, that there exists some $\psi\in X_0(0)$ such that
$$\limsup_{t\rightarrow\infty}\|(I_1(t,\cdot,\psi),I_{v1}(t,\cdot,\psi),I_2(t,\cdot,\psi),I_{v2}(t,\cdot,\psi))-E_0\|<\delta.$$
Then there exists a $t_1>0$ such that $$0<I_i(t,x,\psi)<\delta,~~~0<I_{vi}(t,x,\psi)<\delta,~~~i=1,2$$
for all $t\geq t_1$ and $x\in\bar{\Omega}$. Then $I_1(t,\cdot,\psi)$ and $I_{v1}(t,\cdot,\psi)$ satisfy
\begin{equation*}
\left\{
\begin{split}
\frac{\partial I_1(t,x)}{\partial t}\geq &D_h\Delta I_1(t,x)-(d+\gamma_1)I_1(t,x)\\
&+\frac{c_1\beta(t,x)l(N(x)-2\delta)}{2p\delta+lN(x)}I_{v1}(t,x),~~~&&t\geq t_1,~x \in \Omega,\\
\frac{\partial I_{v1}(t,x)}{\partial t}\geq &D_v\Delta I_{v1}(t,x)-\eta(t,x)I_{v1}(t,x)\\
&+\frac{\alpha_1\beta(t,x)p(M^*(t,x)-2\delta)}{2p\delta+lN(x)}I_1(t,x),~~~&&t\geq t_1,~x \in \Omega,\\
\frac{\partial I_1(t,x)}{\partial \nu}=&\frac{\partial I_{v1}(t,x)}{\partial \nu}=0,~~~&&t\geq t_1,~x\in\partial\Omega.\\
\end{split}
\right.
\end{equation*}
Let $P_1^\delta: X(0)\rightarrow X(0)$ be the Poincar\'{e} map associated with the following system:
\begin{equation}
\label{compare 2}
\left\{
\begin{split}
\frac{\partial I_1(t,x)}{\partial t}=&D_h\Delta I_1(t,x)-(d+\gamma_1)I_1(t,x)\\
&+\frac{c_1\beta(t,x)l(N(x)-2\delta)}{2p\delta+lN(x)}I_{v1}(t,x),~~~&&t\geq 0,~x \in \Omega,\\
\frac{\partial I_{v1}(t,x)}{\partial t}=&D_v\Delta I_{v1}(t,x)-\eta(t,x)I_{v1}(t,x)\\
&+\frac{\alpha_1\beta(t,x)p(M^*(t,x)-2\delta)}{2p\delta+lN(x)}I_1(t,x),~~~&&t\geq 0,~x \in \Omega,\\
\frac{\partial I_1(t,x)}{\partial \nu}=&\frac{\partial I_{v1}(t,x)}{\partial \nu}=0,~~~&&t\geq 0,~x\in\partial\Omega.
\end{split}
\right.
\end{equation}
In view of Lemma \ref{same sign R_i}, we have that $\mathcal{R}_1>1$ is equivalent to $r(P_1)>1$. By continuity, we see that $\lim_{\delta\rightarrow 0}r(P_1^\delta)=r(P_1)>1$. Thus, we can fix a sufficiently small number $\delta>0$ such that
$$0<\delta<\min\{\min_{(t,x)\in[0,\omega]\times\bar{\Omega}}\frac{M^*(t,x)}{2},~\min_{x\in\bar{\Omega}}\frac{N(x)}{2}\}~~\mbox{and}~~r(P_1^\delta)>1.$$
Since $P_1^\delta$ is compact and strongly positive on $E(0)$, then Krein-Rutman Theorem implies that $r(P_1^\delta)$ is a simple eigenvalue of $P_1^\delta$ having a strongly positive eigenvector.
It then follows from \cite[Theorem 2.16 and Remark 2.20]{Liang et al. 2017} that there is a positive $\omega$-periodic function $\vartheta^\delta(t,x)$ such that $e^{\mu^\delta t}\vartheta^\delta(t,x)$ is a positive solution of \eqref{compare 2}, where $\mu^\delta=\frac{\ln r(P_1^\delta)}{\omega}>0$.
From Lemma \ref{solution}, we know  that
\[I_1(t,x,\psi)>0~~\mbox{and}~~ I_{v1}(t,x,\psi)> 0,~~~\forall t>0, x\in \bar{\Omega}.\]
Thus, we may choose a $c>0$ such that
\[(I_1(t_1,x,\psi),I_{v1}(t_1,x,\psi))\geq c e^{\mu^\delta t_1}\vartheta^\delta(t_1,x),~~~\forall x\in \bar{\Omega}.\] A simple comparison leads to
$$(I_1(t,\cdot,\psi),I_{v1}(t,\cdot,\psi))\geq ce^{\mu^\delta t}\vartheta^\delta(t,x),~~~\forall t\geq t_1,~x\in\bar{\Omega},$$
Since $\mu^\delta>0$, it follows that
$$\lim_{t\rightarrow\infty}I_1(t,x,\psi)=\infty,~\lim_{t\rightarrow\infty}I_{v1}(t,x,\psi)=\infty~~{\rm uniformly~for}~x\in\bar{\Omega}.$$
By performing a similar analysis on $(I_2(t,x,\psi),I_{v2}(t,x,\psi))$, when $\mathcal{R}_2>1$,
$$\lim_{t\rightarrow\infty}I_2(t,x,\psi)=\infty,~\lim_{t\rightarrow\infty}I_{v2}(t,x,\psi)=\infty~~{\rm uniformly~for}~x\in\bar{\Omega}.$$
This contradicts the boundedness of $I_i$ and $I_{vi}$, $i=1,2$.
\end{proof}

\begin{theorem}
\label{uniformly persistent}
Suppose that $\mathcal{R}_1>1,~\mathcal{R}_2>1,~\hat{\mathcal{R}}_1>1,$ and $\hat{\mathcal{R}}_2>1$, then system \eqref{model} admits at least one positive $\omega$-periodic solution, and there exists a constant $\delta^*>0$ such that for any $\varphi\in X_0(0)$, we have
$$\liminf_{t\rightarrow \infty}\min_{x\in\bar{\Omega}}I_i(t,x)\geq\delta^* ,~~~\liminf_{t\rightarrow \infty}\min_{x\in\bar{\Omega}}I_{vi}(t,x)\geq\delta^*,~~~i=1,2.$$
\end{theorem}

\begin{proof}
For any $\varphi\in X_0(0)$, by Lemma 2.2, we have
$$I_i(t,x,\varphi)>0,~I_{vi}(t,x,\varphi)>0,~~~i=1,2,~\forall t>0,~x\in\bar{\Omega}.$$
Thus, $Q^n(X_0(0))\subset X_0(0),~\forall n\in\mathbb{N}$. Furthermore, $Q$ admits a global attractor on $X(0)$.

Next we prove that $Q$ is uniformly persistent with respect to $(X_0(0), \partial X_0(0))$.
Recalling the definitions of $E_0, E_1(t,x), E_2(t,x)$ in Section \ref{Section 3.2}, we let
\[M_0=E_0, ~~M_1=E_1(0,\cdot),~~M_2=E_2(0,\cdot).\]
Then we have the following claims.

{\bf Claim 1.} There is a $\delta_1>0$ such that
$$\limsup_{n\rightarrow\infty}\|Q^n(\varphi)-M_0\|\geq\delta_1,~~~\forall\varphi\in X_0(0).$$
This claim directly follows from Lemma \ref{weak repeller}.

Consider an auxiliary system with parameter $\varepsilon$:
\begin{equation}
\label{compare 3}
\left\{
\begin{split}
\frac{\partial I_2(t,x)}{\partial t}=&D_h\Delta I_2(t,x)-(d+\gamma_2)I_2(t,x)\\
&+\frac{c_2\beta(t,x)l(N(x)-I_1^*(t,x)-2\varepsilon)}{p(I_1^*(t,x)+2\varepsilon)+l(N(x)-I_1^*(t,x)-2\varepsilon)}I_{v2}(t,x),\\
\frac{\partial I_{v2}(t,x)}{\partial t}=&D_v\Delta I_{v2}(t,x)-\eta(t,x)I_{v2}(t,x)\\
&+\frac{\alpha_2\beta(t,x)p(M^*(t,x)-I_{v1}^*(t,x)-2\varepsilon)}{p(I_1^*(t,x)+2\varepsilon)+l(N(x)-I_1^*(t,x)-2\varepsilon)}I_2(t,x),\\
\frac{\partial I_2(t,x)}{\partial \nu}=&\frac{\partial I_{v2}(t,x)}{\partial \nu}=0,\\
\end{split}
\right.
\end{equation}
for all $t>0$. Let $\hat{P}_2^\varepsilon:=\hat{P}_2^\varepsilon(\omega)$ be the Poincar\'{e} map of \eqref{compare 3}.  Since $\lim_{\varepsilon\rightarrow 0}r(\hat{P}_2^\varepsilon)=r(\hat{P}_2)>1$, we can fix a small number $\varepsilon>0$ such that $r(\hat{P}_2^\varepsilon)>1$.  As discussed in Lemma \ref{weak repeller}, there is a positive $\omega$-periodic function $\vartheta^\varepsilon(t,x)$ such that $e^{\mu^\varepsilon t}\vartheta^\varepsilon(t,x)$ is a positive solution of \eqref{compare 3}, where $\mu^\varepsilon=\frac{\ln r(\hat{P}_2^\varepsilon)}{\omega}>0$. For $\varepsilon>0$ above, by the continuous dependence of solutions on the initial value, there exists $\delta_2=\delta_2(\varepsilon)>0$ such that for all $\varphi\in X_0(0)$ with $\|\varphi-M_1\|\leq\delta_2$, we have
$\|Q(t)\varphi-Q(t)M_1\|<\varepsilon,~\forall t\in[0,\omega]$.

{\bf Claim 2.} $\limsup_{n\rightarrow\infty}\|Q^n(\varphi)-M_1\|\geq\delta_2,~\forall\varphi\in X_0(0).$

Suppose the claim is false, then $\limsup_{n\rightarrow\infty}\|Q^n(\psi)-M_1\|<\delta_2$ for some $\psi\in X_0(0)$. Then there exists an integer $N_1\geq1$ such that $\|Q^n(\psi)-M_1\|<\delta_2$ for $n\geq N_1$. For any $t\geq N_1\omega$, letting $t=n\omega+\tilde{t}$ with $n=[t/\omega]$ and $\tilde{t}\in[0,\omega)$, we have
\begin{equation*}
\label{3}
\|Q(t)\psi-Q(t)M_1\|=\|Q(\tilde{t})(Q^n(\psi))-Q(\tilde{t})M_1\|<\varepsilon.
\end{equation*}
According to the above inequality and Lemma \ref{solution}, we infer that
\begin{align*}
& 0<I_1(t,x,\psi)<I_1^*(t,x,\psi)+\varepsilon,~0<I_{v1}(t,x,\psi)<I_{v1}^*(t,x,\psi)+\varepsilon,\\
& 0<I_2(t,x,\psi)<\varepsilon,~0<I_{v2}(t,x,\psi)<\varepsilon,~~~\forall t\geq N_1\omega,~x\in\bar{\Omega}.
\end{align*}
As a result, $I_2(t,x,\psi)$ and $I_{v2}(t,x,\psi)$ satisfy
\begin{equation}
\label{compare 4}
\left\{
\begin{split}
\frac{\partial I_2(t,x)}{\partial t}\geq &D_h\Delta I_2(t,x)-(d+\gamma_2)I_2(t,x)\\
&+\frac{c_2\beta(t,x)l(N(x)-I_1^*(t,x)-2\varepsilon)}{p(I_1^*(t,x)+2\varepsilon)+l(N(x)-I_1^*(t,x)-2\varepsilon)}I_{v2}(t,x),\\
\frac{\partial I_{v2}(t,x)}{\partial t}\geq &D_v\Delta I_{v2}(t,x)-\eta(t,x)I_{v2}(t,x)\\
&+\frac{\alpha_2\beta(t,x)p(M^*(t,x)-I_{v1}^*(t,x)-2\varepsilon)}{p(I_1^*(t,x)+2\varepsilon)+l(N(x)-I_1^*(t,x)-2\varepsilon)}I_2(t,x),\\
\frac{\partial I_2(t,x)}{\partial \nu}=&\frac{\partial I_{v2}(t,x)}{\partial \nu}=0,\\
\end{split}
\right.
\end{equation}
for all $t\geq N_1\omega$. Since $\psi\in X_0(0)$, $I_2(t,x,\psi)>0$ and $I_{v2}(t,x,\psi)>0$ for all $t\geq 0$ and $x\in \bar{\Omega}$, there exists a $k>0$ such that  \[(I_2(N_1\omega,x,\psi),I_{v2}(N_1\omega,x,\psi))\geq k e^{\mu^\varepsilon N_1\omega}\vartheta^\varepsilon(N_1\omega,x),~~~x\in\bar{\Omega}.\]
An application of the comparison theorem to \eqref{compare 4} yields
$$(I_2(t,x,\psi),I_{v2}(t,x,\psi))\geq k e^{\mu^\varepsilon t}\vartheta^\varepsilon(t,x),~~~\forall t\geq N_1\omega,~x\in\bar{\Omega}.$$
Since $\mu^\varepsilon>0$, one sees that $I_2(t,\cdot,\psi)\rightarrow\infty,~I_{v2}(t,\cdot,\psi)\rightarrow\infty$ as $t\rightarrow\infty$. This gives rise to a contradiction, and thereby, the above claim is true.

In a similar way, we can prove the following claim.

{\bf Claim 3.} There exists a $\delta_3>0$ such that
$$\limsup_{n\rightarrow\infty}\|Q^n(\varphi)-M_2\|\geq\delta_3,~\forall\varphi\in X_0(0).$$

With the above three claims, we see that $M_0, M_1$ and $M_2$ are isolated invariant sets for $Q$ in $X(0)$ and $W^s(M_i)\bigcap X_0(0)=\emptyset, i=0,1,2$, where $W^s(M_i)$ is the stable set of $M_i$ for $Q$. Set
$$M_\partial:=\{\varphi\in\partial X_0(0):~Q^n(\varphi)\in\partial X_0(0),~\forall n\in\mathbb{N}\}.$$
We now show that $M_\partial=\mathcal {M}_0$, where
\[\mathcal {M}_0:=\{\varphi\in\partial X_0(0):~\varphi_1(\cdot)=\varphi_2(\cdot)\equiv 0~\mbox{or}~\varphi_3(\cdot)=\varphi_4(\cdot)\equiv 0\}.\]
Obviously, it suffices to prove $M_\partial \subset \mathcal {M}_0$.
For any given $\varphi\in M_\partial$, we have $I_1(n\omega,\cdot,\varphi)\equiv 0$ or $I_{v1}(n\omega,\cdot,\varphi)\equiv 0$ or $I_2(n\omega,\cdot,\varphi)\equiv 0$ or $I_{v2}(n\omega,\cdot,\varphi)\equiv 0, \forall n\in\mathbb{N}$. Assume that $\varphi\notin \mathcal {M}_0$, then there are eight possibilities as below:
\begin{itemize}
\item [(i)] $\varphi_1(\cdot)=I_1(0,\cdot,\varphi)\equiv 0,~\varphi_2(\cdot)=I_{v1}(0,\cdot,\varphi)>0,~\varphi_3(\cdot)=I_2(0,\cdot,\varphi)\equiv 0,~\varphi_4(\cdot)=I_{v2}(0,\cdot,\varphi)>0$.
\item[(ii)]  $\varphi_1(\cdot)=I_1(0,\cdot,\varphi)\equiv 0,~\varphi_2(\cdot)=I_{v1}(0,\cdot,\varphi)>0,~\varphi_3(\cdot)=I_2(0,\cdot,\varphi)>0,~\varphi_4(\cdot)=I_{v2}(0,\cdot,\varphi)\equiv 0$.
\item [(iii)]  $\varphi_1(\cdot)=I_1(0,\cdot,\varphi)\equiv 0,~\varphi_2(\cdot)=I_{v1}(0,\cdot,\varphi)>0,~\varphi_3(\cdot)=I_2(0,\cdot,\varphi)>0,~\varphi_4(\cdot)=I_{v2}(0,\cdot,\varphi)>0$.

\item[(iv)] $\varphi_1(\cdot)=I_1(0,\cdot,\varphi)>0,~\varphi_2(\cdot)=I_{v1}(0,\cdot,\varphi)\equiv 0,~\varphi_3(\cdot)=I_2(0,\cdot,\varphi)\equiv 0,~\varphi_4(\cdot)=I_{v2}(0,\cdot,\varphi)>0$.

\item[(v)]$\varphi_1(\cdot)=I_1(0,\cdot,\varphi)>0,~\varphi_2(\cdot)=I_{v1}(0,\cdot,\varphi)\equiv 0,~\varphi_3(\cdot)=I_2(0,\cdot,\varphi)>0,~\varphi_4(\cdot)=I_{v2}(0,\cdot,\varphi)\equiv 0$.

\item [(vi)] $\varphi_1(\cdot)=I_1(0,\cdot,\varphi)>0,~\varphi_2(\cdot)=I_{v1}(0,\cdot,\varphi)\equiv 0,~\varphi_3(\cdot)=I_2(0,\cdot,\varphi)>0,~\varphi_4(\cdot)=I_{v2}(0,\cdot,\varphi)>0$.

\item [(vii)]$\varphi_1(\cdot)=I_1(0,\cdot,\varphi)>0,~\varphi_2(\cdot)=I_{v1}(0,\cdot,\varphi)>0,~\varphi_3(\cdot)=I_2(0,\cdot,\varphi)\equiv 0,~\varphi_4(\cdot)=I_{v2}(0,\cdot,\varphi)>0$.

\item [(viii)]    $\varphi_1(\cdot)=I_1(0,\cdot,\varphi)>0,~\varphi_2(\cdot)=I_{v1}(0,\cdot,\varphi)>0,~\varphi_3(\cdot)=I_2(0,\cdot,\varphi)>0,~\varphi_4(\cdot)=I_{v2}(0,\cdot,\varphi)\equiv 0$.
\end{itemize}
By Lemma 2.2, in case (i), we obtain that $I_{v1}(t,x,\varphi)>0, I_{v2}(t,x,\varphi)>0$ for all $ t>0$ and $x\in\bar{\Omega}$. Further, using the first and third equation of \eqref{model}, one obtains that $I_1(t,x,\varphi)>0,~I_2(t,x,\varphi)>0, \forall t>0, x\in \bar{\Omega}$, which contradicts with the fact $\varphi\in M_\partial$.
By performing a similar analysis, we can show that (ii)-(viii) are impossible. Hence,
$\varphi\in \mathcal {M}_0$, and hence, $M_\partial\subset \mathcal {M}_0$. This proves $M_\partial=\mathcal {M}_0$.

Let $\omega(\varphi)$ be the omega limit set of the forward orbit $\gamma^+(\varphi):=\{Q^n(\varphi):~\forall n\in\mathbb{N}\}$. We further have the following claims.

{\bf Claim 4.} $\cup_{\varphi\in M_\partial}\omega(\varphi)\subset \{M_0, M_1, M_2\}$.

Obviously, there are three possibilities for $\varphi\in M_\partial=\mathcal {M}_0:$
\begin{align*}
& \mbox{\bf Case 1}: ~~\varphi_1(\cdot)=\varphi_2(\cdot)\equiv0,~~ \varphi_3(\cdot)>0~\mbox{or}~\varphi_4(\cdot)> 0;\\
& \mbox{\bf Case 2}: ~~\varphi_3(\cdot)=\varphi_4(\cdot)\equiv0,~~ \varphi_1(\cdot)> 0~\mbox{or}~\varphi_2(\cdot)> 0; \\
& \mbox{\bf Case 3}: ~~\varphi_1(\cdot)=\varphi_2(\cdot)=\varphi_3(\cdot)=\varphi_4(\cdot)\equiv0.
\end{align*}
In what follows, we aim to show that claim 4 holds for each of the above three cases.

If {\bf Case 1} happens, then $I_1(t,x,\varphi)=I_{v1}(t,x,\varphi)\equiv 0$ for all $t\geq 0$ and $x\in\bar{\Omega}$. In view of system \eqref{model}, $I_2(t,x,\varphi),I_{v2}(t,x,\varphi)$ satisfy system \eqref{single-strain model} with $i=2, j=1$. Since $\mathcal{R}_2>1$, it follows from Theorem 3.1 that \[\lim_{t\rightarrow\infty} \big(I_2(t,x,\varphi)-I_2^*(t,x)\big)=0, ~~ \lim_{t\rightarrow\infty}\big(I_{v2}(t,x,\varphi)-I_{v2}^*(t,x)\big)=0\] uniformly for $x\in\bar{\Omega}$. Hence, $\omega(\varphi)=M_2$ for any $\varphi\in M_\partial$.

For {\bf Case 2}, by repeating arguments similar to {\bf Case 1},
we can show that $\omega(\varphi)=M_1$ for any $\varphi\in M_\partial$. For {\bf Case 3}, one immediately finds that
$$(I_1(t,x,\varphi),I_{v1}(t,x,\varphi),I_2(t,x,\varphi),I_{v2}(t,x,\varphi))=(0,0,0,0),~~~\forall t\geq 0, x\in \bar{\Omega}.$$
This implies that $\omega(\varphi)=M_0$ for any $\varphi\in M_\partial$. Thus claim 4 is obtained.

{\bf Claim 5.} $M_1$ and $M_2$ are locally stable, and $M_0$ is unstable for $Q$ in $M_\partial$.

Suppose that $\varphi\in M_\partial$, we have that $M_\partial=M_\partial^1 \cup M_\partial^2$, where
\[M_\partial^1:=\{\varphi\in \partial X_0(0):~\varphi_1(\cdot)=\varphi_2(\cdot)\equiv0\},~~M_\partial^2:=\{\varphi\in \partial X_0(0):~\varphi_3(\cdot)=\varphi_4(\cdot)\equiv0\}.\]
If $\varphi\in M_\partial^1$, then system \eqref{model} restricted on $M_\partial^1$ is a monotone system. Thus, $M_2$ is locally Lyapunov stable for $Q$ in $M_\partial^1$ due to \cite[Lemma 2.2.1]{Zhao 2017b}, and $M_0$ is unstable in $M_\partial^1$. In a similar manner, if $\varphi\in M_\partial^2$, we can prove that
$M_1$ is locally Lyapunov stable for $Q$ in $M_\partial^2$, and $M_0$ is unstable in $M_\partial^2$.

The claim 5 implies that no subset of $\{M_0,M_1,M_2\}$ forms a cycle in $\partial X_0(0)$. Based on the above analysis, it follows from the acyclicity theorem on uniform persistence for maps \cite[Theorem 1.3.1 and Remark 1.3.1]{Zhao 2017b} that $Q: X(0)\rightarrow X(0)$ is uniformly persistent with respect to $(X_0(0),\partial X_0(0))$ in the sense that there exists $\tilde{\delta}>0$ such that
$$\liminf_{n\rightarrow\infty}d(Q^n(\varphi),\partial X_0(0))\geq\tilde{\delta},~~~\forall \varphi\in X_0(0).$$
By \cite[Theorem 4.5]{Magal and Zhao 2005} with $\rho(\phi)=d(\phi,\partial X_0(0))$, $Q$ admits a global attractor $A_0$ in $X_0(0)$, and $Q$ has a fixed point in $\varphi^*\in A_0$. Clearly, $u(t,\cdot,\varphi^*)$ is an $\omega$-periodic solution of \eqref{model} and it is strictly positive due to Lemma \ref{solution}.

Finally, we use the arguments in \cite[Section 11.2]{Zhao 2017b} to obtain the practical uniform persistence. Since $A_0=Q(\omega)A_0=Q(A_0)$, we have $\varphi_i>0, i=1,2,3,4, \forall \varphi\in A_0$. Let $B_0:=\cup_{t\in[0,\omega]}Q(t)A_0$. Then \cite[Theorem 3.1.1]{Zhao 2017b} implies that $B_0\subset X_0(0)$, and $\lim_{t\rightarrow\infty}d(Q(t)\varphi,B_0)=0$ for all $\varphi\in X_0(0).$ Define a continuous function $p: X(0)\rightarrow\mathbb{R}_+$ by
$$p(\varphi):=\min\{\min_{x\in\bar{\Omega}}\varphi_1(x),~\min_{x\in\bar{\Omega}}\varphi_2(x),\min_{x\in\bar{\Omega}}\varphi_3(x),\min_{x\in\bar{\Omega}}\varphi_4(x)\},~\forall\varphi\in X(0).$$
Since $B_0$ is compact, it follows that $\inf_{\varphi\in B_0}p(\varphi)=\min_{\varphi\in B_0}p(\varphi)>0$. Therefore, there exists a $\delta^*>0$ such that for any $\varphi\in X_0(0)$,
\begin{align*}
& \liminf_{t\rightarrow\infty}{\rm min}_{x\in\bar{\Omega}}I_i(t,x,\varphi)\geq\delta^*,~~~\liminf_{t\rightarrow\infty}{\rm min}_{x\in\bar{\Omega}}I_{vi}(t,x,\varphi)\geq\delta^*,~~~i=1,2.
\end{align*}
The proof is complete.
\end{proof}

\section{Asymptotic behavior of $\mathcal{R}_0$}

In this section, we use the recent theory developed in \cite{Zhang and Zhao 2021} to
study the asymptotic behavior of the basic reproduction number
as the diffusion coefficients go to zero and infinity. To do this, we write $$\mathcal{R}_0(D_h,D_v)=\max\{\mathcal{R}_1(D_h,D_v),\mathcal{R}_2(D_h,D_v)\}.$$
Observe that for each $x\in\bar{\Omega}$, the equation
$$\frac{\partial M(t,x)}{\partial t}=\Lambda(t,x)-\eta(t,x)M(t,x)$$
admits a globally stable positive $\omega$-periodic solution $M_0(t,x)$, and it is continuous on $\R\times \bar{\Omega}$. Define $\widetilde{g}(t):=|\Omega|^{-1}\int_{\Omega}g(t,x)dx$. One immediately sees that the following scalar periodic equation
$$\frac{\partial M(t,x)}{\partial t}=\widetilde{\Lambda}(t)-\widetilde{\eta}(t)M(t,x)$$
has a unique positive $\omega$-periodic solution
\[\widetilde{M}_{\infty}(t)=\left[\int_0^t \widetilde{\Lambda}(s)e^{\int_0^s \widetilde{\eta}(\xi)d\xi}ds+
\frac{ \int_0^\omega \widetilde{\Lambda}(s)e^{\int_0^s\widetilde{\eta}(\xi)d\xi}ds}{e^{\int_0^\omega \widetilde{\eta}(s)ds}-1}\right] e^{-\int_0^t \widetilde{\eta}(s)ds},\]
which is globally asymptotically stable.

It is easy to verify that assumptions (H1)-(H5) in \cite{Zhang and Zhao 2021} are valid.
An direct application of \cite[Theorems 5.2 and 5.5]{Zhang and Zhao 2021} leads to
$$\lim_{D_v\rightarrow 0}\|M^*(t,\cdot)-M_0(t,\cdot)\|_\mathbb{Y}=0,~~\lim_{D_v\rightarrow \infty}\|M^*(t,\cdot)-\widetilde{M}_\infty(t)\|_\mathbb{Y}=0$$
hold uniformly on $t\in\R$. For each $x\in\bar{\Omega}$, let $\{\Gamma^i_{x,0}(t,s):t\geq s\}~(i=1,2)$ be the evolution family on $\mathbb{R}^2$ associated with the following system:
\begin{equation*}
\label{without diffusion 1}
\left\{
\begin{split}
\frac{\partial I_i(t,x)}{\partial t}&=-(d+\gamma_i)I_i(t,x),~~~&&t>s,~x\in\bar{\Omega},\\
\frac{\partial I_{vi}(t,x)}{\partial t}&=-\eta(t,x)I_{vi}(t,x),~~~&&t>s,~x\in\bar{\Omega},
\end{split}
\right.
\end{equation*}
and define
$$
F^i_0(t,x)
\left(
\begin{array}{c}
\psi_{1}\\
\psi_{2}
\end{array}
\right)
=
\left(
\begin{array}{c}
c_i\beta(t,x)\psi_2\\
\frac{\alpha_i\beta(t,x)pM_0(t,x)}{lN(x)}\psi_1
\end{array}
\right),~\forall t\in\mathbb{R},~x\in\bar{\Omega},~\psi=(\psi_1,~\psi_2)\in\mathbb{R}^2.
$$
Let $\{\widetilde{\Gamma}^i_{\infty}(t,s):t\geq s\}~(i=1,2)$ be the evolution family on $\mathbb{R}^2$ of the following system:
\begin{equation*}
\label{without diffusion 2}
\left\{
\begin{split}
\frac{\partial I_i(t,x)}{\partial t}&=-(d+\gamma_i)I_i(t,x),~~~&& t>s,~x\in\bar{\Omega},\\
\frac{\partial I_{v2}(t,x)}{\partial t}&=-\widetilde{\eta}(t)I_{vi}(t,x),~~~&&t>s, ~x\in\bar{\Omega},
\end{split}
\right.
\end{equation*}
and define
$$
\widetilde{F}^i_{\infty}(t)
\left(
\begin{array}{c}
\psi_{1}\\
\psi_{2}
\end{array}
\right)
=
\left(
\begin{array}{c}
c_i\widetilde{\beta}(t)\psi_2\\
\widetilde{f}^i_{21}(t)\widetilde{M}_\infty(t)\psi_1
\end{array}
\right),~~~\forall t\in\mathbb{R},~\psi=(\psi_1,~\psi_2)\in\mathbb{R}^2,
$$
where
$$\widetilde{f}^i_{21}(t):=|\Omega|^{-1}\int_{\Omega}\frac{\alpha_i\beta(t,x)p}{lN(x)}dx,~~~\forall t\in\mathbb{R}.$$

Let $C_\omega(\mathbb{R},\mathbb{R}^2)$ be the Banach space of all continuous and $\omega$-periodic functions from $\mathbb{R}$ to $\mathbb{R}^2$, which is endowed with the maximum norm. For each $x\in\bar{\Omega}$, we respectively define bounded linear positive operators $L^i_{x,0}$ and $\widetilde{L}^i_{\infty}, i=1,2$, on $C_\omega(\mathbb{R},\mathbb{R}^2)$ by
$$[L^i_{x,0}v](t):=\int^{\infty}_{0}\Gamma^i_{x,0}(t,t-s)F^i_0(t-s,x)v(t-s)ds,~~~\forall t\in\mathbb{R},~v\in C_\omega(\mathbb{R},\mathbb{R}^2),$$
and
$$[\widetilde{L}^i_{\infty}v](t):=\int^{\infty}_{0}\widetilde{\Gamma}^i_{\infty}(t,t-s)\widetilde{F}^i_{\infty}(t-s)v(t-s)ds,~~~\forall t\in\mathbb{R},~v\in C_\omega(\mathbb{R},\mathbb{R}^2).$$

Then we define $R_i(x,0):=r(L^i_{x,0}),~\forall x\in\bar{\Omega},~i=1,2$ and $\widetilde{R}_i(\infty):=r(\widetilde{L}^i_{\infty}),~i=1,2$. By \cite[Theorem 4.1]{Zhang and Zhao 2021} with $\kappa={\rm diag}(D_h,D_v),~\chi=D_v$ and $\chi_0=0$, and $\kappa={\rm diag}(D_h,D_v),~\chi=\frac{1}{D_v}$ and $\chi_0=0$, respectively, it follows that
$$\lim_{\max(D_h,D_v)\rightarrow 0}\mathcal{R}_i(D_h,D_v)=\max_{x\in\bar{\Omega}}R_i(x,0),$$
$$\lim_{\min(D_h,D_v)\rightarrow \infty}\mathcal{R}_i(D_h,D_v)=\widetilde{R}_i(\infty),~~i=1,2.$$
Therefore,
$$\lim_{\max(D_h,D_v)\rightarrow 0}\mathcal{R}_0(D_h,D_v)=\max_{x\in\bar{\Omega}}\{R_1(x,0),R_2(x,0)\},$$
$$\lim_{\min(D_h,D_v)\rightarrow \infty}\mathcal{R}_0(D_h,D_v)=\max\{\widetilde{R}_1(\infty),\widetilde{R}_2(\infty)\}.$$

\section{Numerical simulations}
To verify these analytic results and examine the effects of seasonality and vector-bias on the malaria transmission, we perform illustrative numerical investigations.

\subsection{Competitive exclusion and coexistence}

We choose the period of our model to be $T=12 ~\mbox{months}$ and concentrate on one dimensional domain $\Omega=[0,\pi]$. For illustrative
purpose, we only let $\beta(t,x)$ be the time-dependent parameters, given
by
\begin{align*}
\beta(t,x)=& 4\times (5.1492-1.83692\cos(0.523599t)-0.175817\cos(1.0472t)\\
&-0.166233\cos(1.5708t)-0.16485\cos(2.0944t)-0.17681\cos(2.61799t)\\
&-1.37079\sin(0.523599t)+0.296267\sin(1.0472t)+0.2134\sin(1.5708t)\\
&-0.295228\sin(2.0944t)-0.201712\sin(2.61799t))~\mbox{month}^{-1},
\end{align*}
which is adapted from \cite{Lou and Zhao 2010}. Unless stated otherwise,
the baseline parameters are seen in Table \ref{parameters}.
We use the numerical scheme proposed in \cite[Lemma 2.5 and Remark 3.2]{Liang et al. 2019} to compute the reproduction number of each strain. In order to demonstrate the outcomes of competitive exclusion and coexistence, we consider the following three cases.
\begin{table}[!ht]
\caption{Parameter values.}\label{parameters}
    \centering
\begin{tabular}{llll}
  \toprule
  Parameter & Value~(range) & Dimension & Reference\\
  \midrule
  $N(x)$  & 110 & dimensionless & \cite{Bai et al. 2018} \\
  $M^*(t,x)$ & 220 & dimensionless & \cite{Shi and Zhao 2021}\\
  $d$ & $1/(72\times 12)$ & month$^{-1}$ & \cite{Bai et al. 2018}\\
  $\eta(t,x)$ & 0.8 & month$^{-1}$ & \cite{Shi and Zhao 2021}\\
  $D_h$ & 0.4 & km$^2\cdot$ month$^{-1}$ & \cite{Bai et al. 2018}\\
  $D_v$ & 0.02 & km$^2\cdot$ month$^{-1}$ & \cite{Bai et al. 2018} \\
  $p$ & 0.8~(0,1) & dimensionless & \cite{Bai et al. 2018} \\
  $l$ & 0.2~(0,1) & dimensionless & \cite{Bai et al. 2018} \\
  \bottomrule
\end{tabular}\label{paraTab}
  \end{table}

{\it Case 1.} $\mathcal{R}_1>1, \mathcal{R}_2>1, \hat{\mathcal{R}}_1>1$ and $\hat{\mathcal{R}}_2>1$. We choose  $\gamma_1= 0.096~\mbox{month}^{-1}, \gamma_2=0.082~\mbox{month}^{-1}, \alpha_1=0.56, \alpha_2=0.6, c_1=0.25, c_2=0.2$. Then we obtain
$\mathcal{R}_1=11.1267$, $\mathcal{R}_2=10.3022$, $\hat{\mathcal{R}}_1=2.2919$, and $\hat{\mathcal{R}}_2=2.2605$. Fig. \ref{long-term-1} shows that
the disease is uniformly persistent, and periodic oscillation phenomenon occurs, which is consistent with Theorem \ref{uniformly persistent}.
\begin{figure}[!ht]
\centering
\subfigure[]{
\includegraphics[height=5.8cm,width=7.2cm,angle=0]{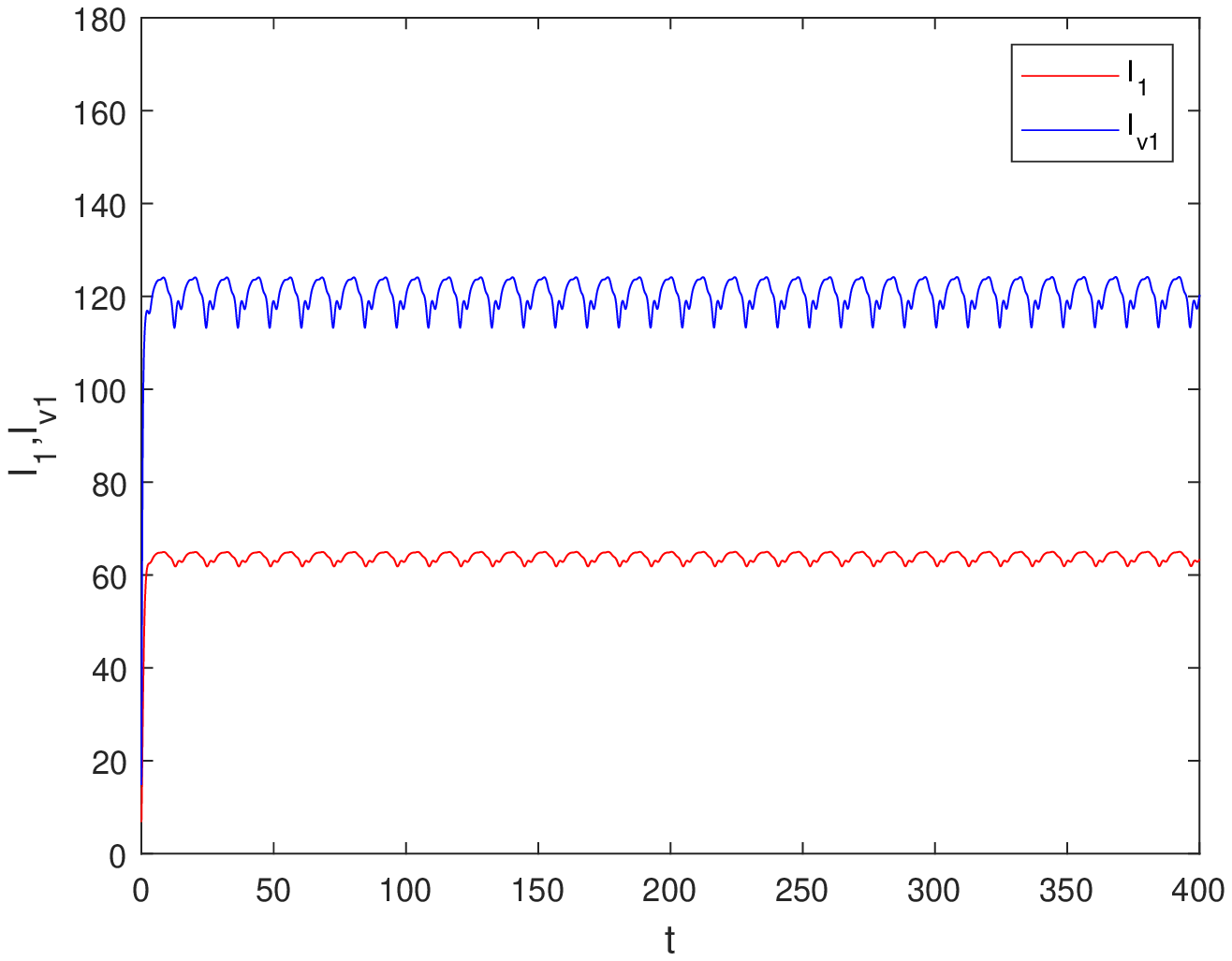}}
\subfigure[]{
\includegraphics[height=5.8cm,width=7.2cm,angle=0]{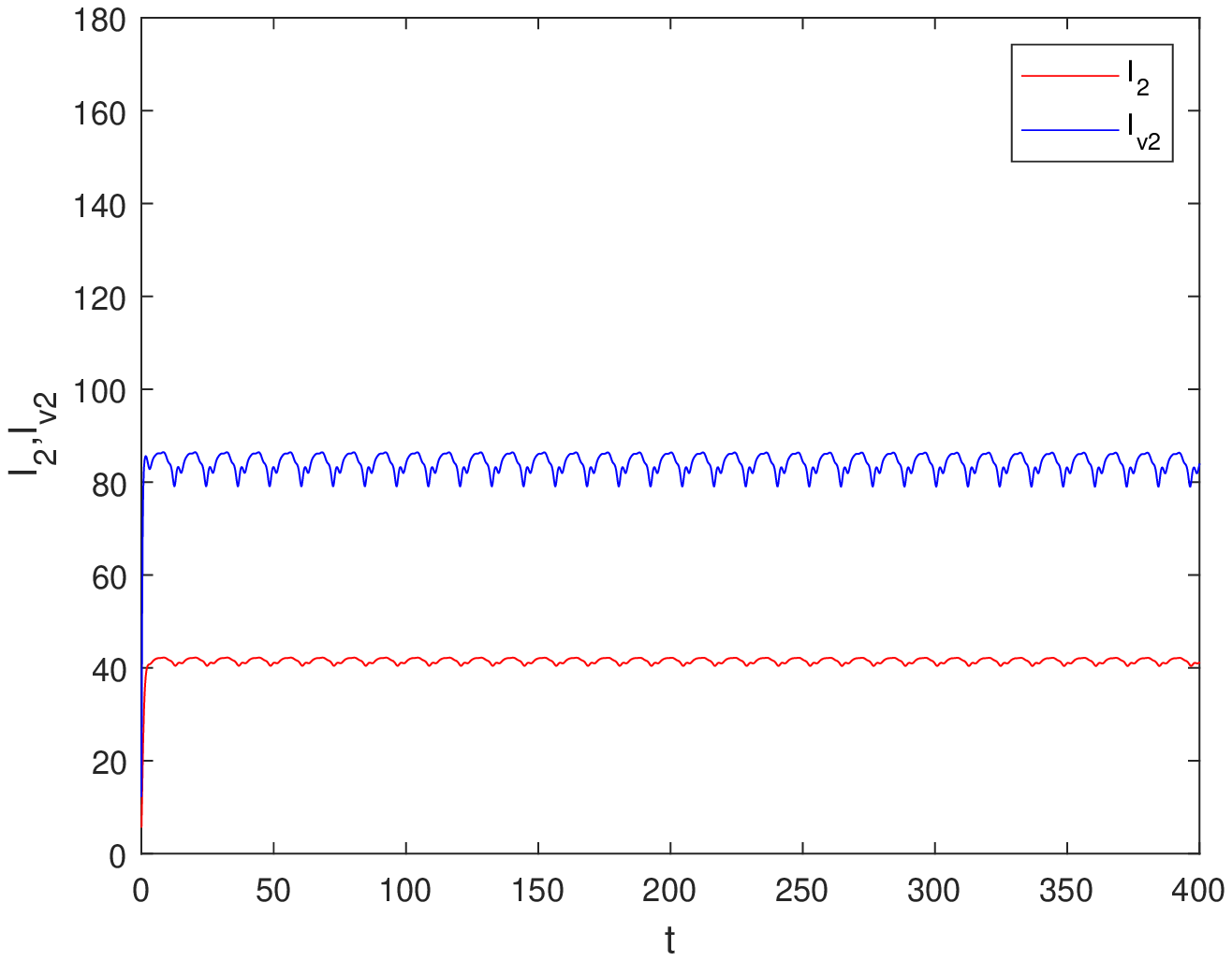}}
\caption{Two strains coexist: (a) the sensitive strains $I_1,I_{v1}$, (b) The resistant strains  $I_2,I_{v2}$. The initial data are chosen as $I_1(0,x)=6(1+\cos(2x)),~I_{v1}(0,x)=10(1+\cos(2x)),~I_2(0,x)=5(1+\cos(2x)),~I_{v2}(0,x)=8(1+\cos(2x)), \forall~x\in[0,\pi]$. }\label{long-term-1}
\end{figure}

{\it Case 2.} $\mathcal{R}_1>1, \mathcal{R}_2>1, \hat{\mathcal{R}}_1>1$ and $\hat{\mathcal{R}}_2<1$.
We choose $\gamma_1=0.083~\mbox{month}^{-1}, \gamma_2=0.082~\mbox{month}^{-1}, \alpha_1=0.35, \alpha_2=0.55, c_1=0.2, c_2=0.1.$
Then we have $\mathcal{R}_1=7.8683$, $\mathcal{R}_2=6.9746$, $\hat{\mathcal{R}}_1=1.0934$, $\hat{\mathcal{R}}_2=0.4729$. Fig. \ref{long-term-2} shows that the sensitive strains are persistent, but the resistant strains die out.
\begin{figure}[!ht]
\centering
\subfigure[]{
\includegraphics[height=5.8cm,width=7.2cm,angle=0]{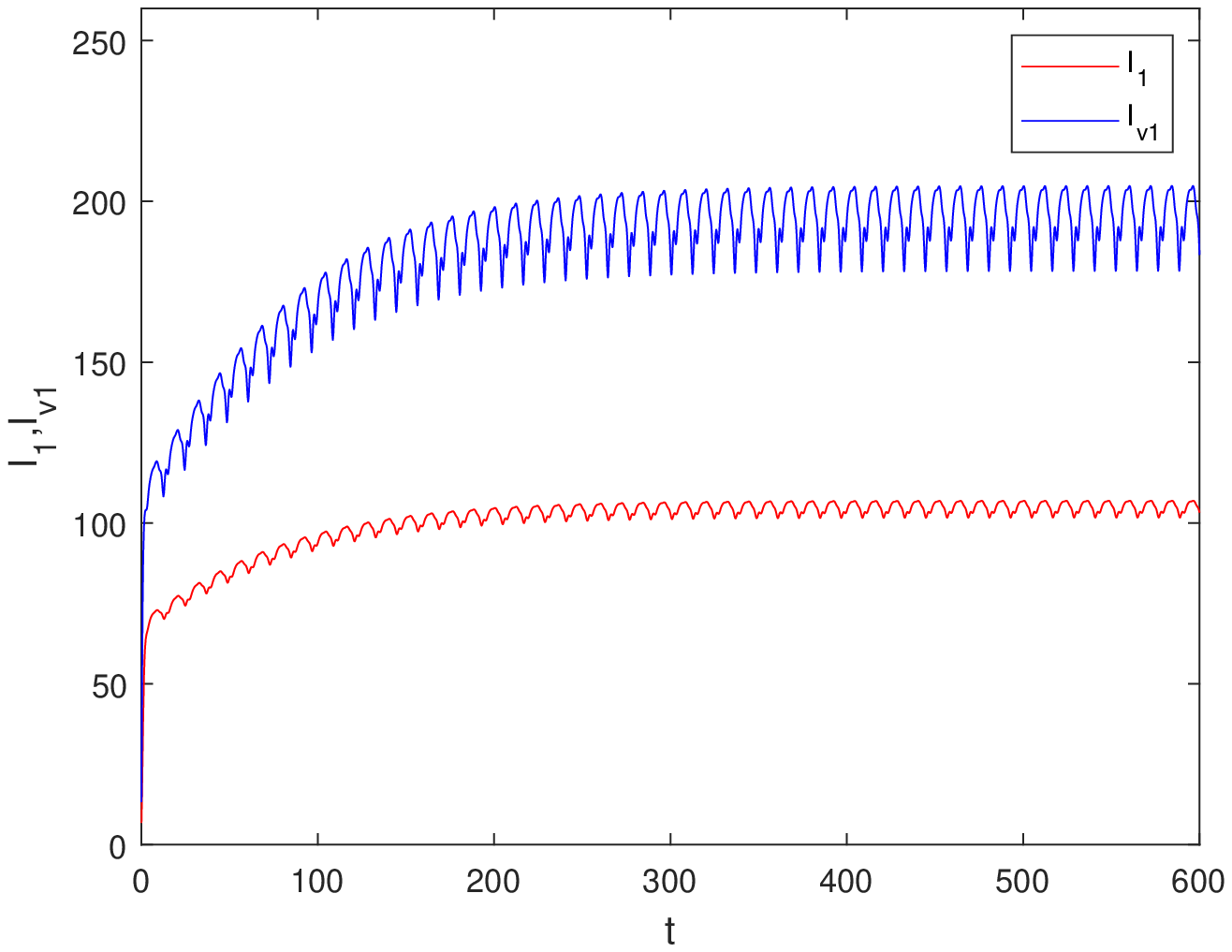}}
\subfigure[]{
\includegraphics[height=5.8cm,width=7.2cm,angle=0]{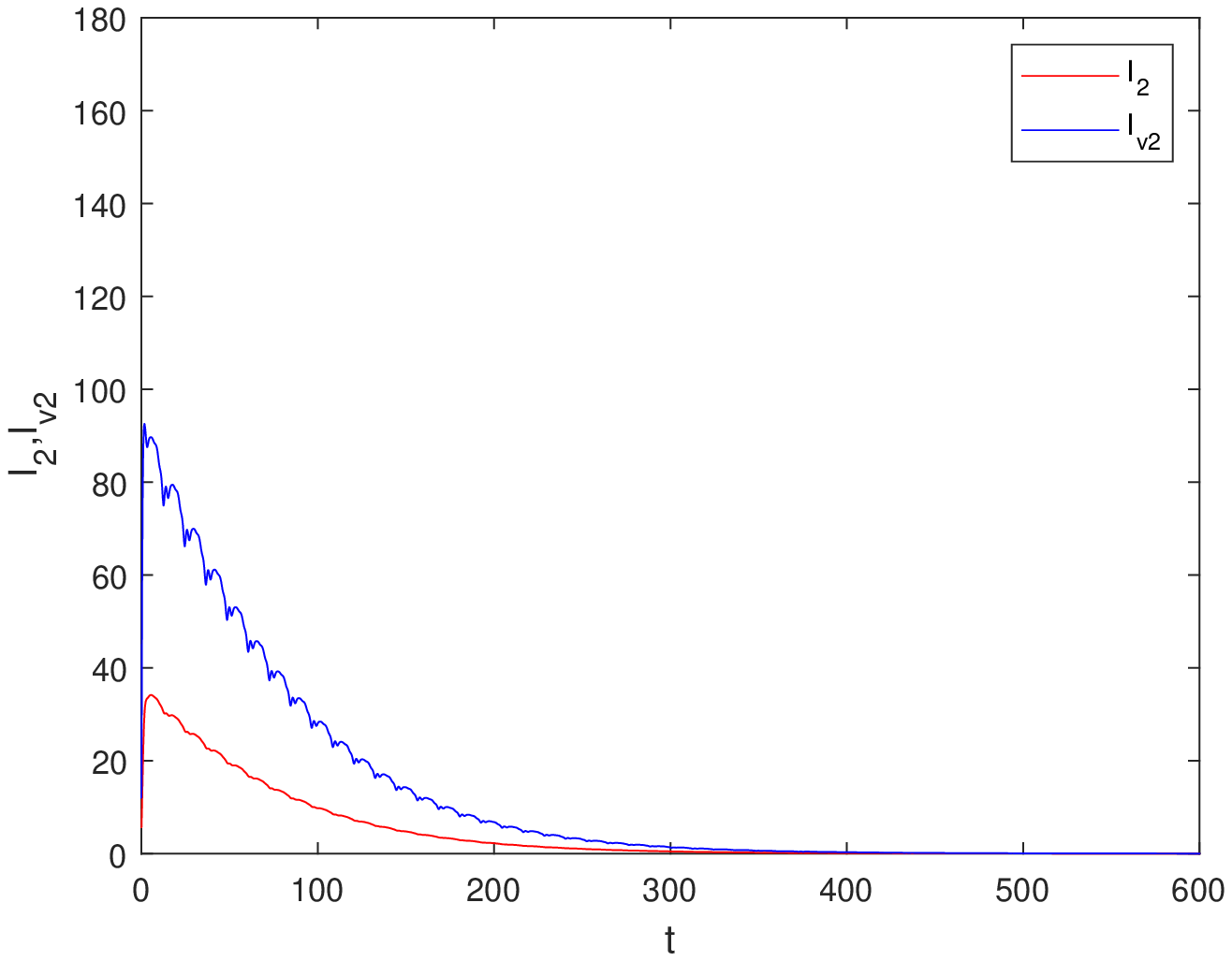}}
\caption{(a) $I_1$ and $I_{v1}$ persist; (b) $I_2$ and $I_{v2}$ die out. The initial data are the same as Fig. \ref{long-term-1}.  }\label{long-term-2}
\end{figure}

{\it Case 3.} $\mathcal{R}_1>1, \mathcal{R}_2>1, \hat{\mathcal{R}}_1<1$ and $\hat{\mathcal{R}}_2>1$.
We choose $\gamma_1=0.096~\mbox{month}^{-1}, \gamma_2=0.082~\mbox{month}^{-1}, \alpha_1=0.55, \alpha_2=0.45, c_1=0.15, c_2=0.2.$
Then we get $\mathcal{R}_1=8.5413$, $\mathcal{R}_2=8.9219$, $\hat{\mathcal{R}}_1=0.7052$, $\hat{\mathcal{R}}_2=1.9601$. Fig. \ref{long-term-3} depicts that the resistant strains persist, but the sensitive strains go extinct.

It should be pointed out that in Figs. \ref{long-term-1}-\ref{long-term-3}, we only plot the graph of $x$-intersection with $x=0.7448$. In addition, for the second and third case, the competitive exclusion phenomena are also observed even though $\mathcal{R}_0>1$.
It is a pity that we now can not prove it, which is left for future
consideration.

\begin{figure}[!ht]
\centering
\subfigure[]{
\includegraphics[height=5.8cm,width=7.2cm,angle=0]{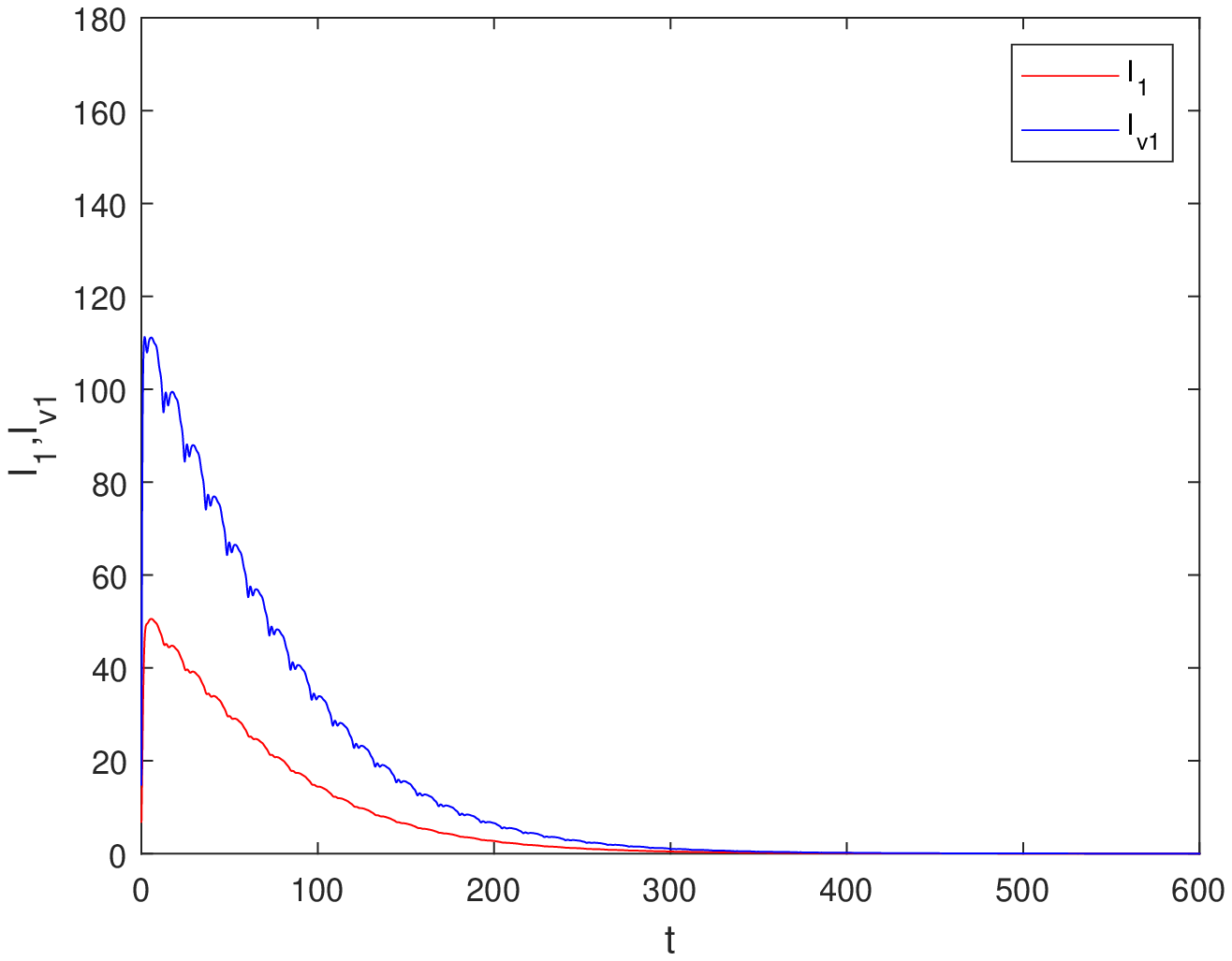}}
\subfigure[]{
\includegraphics[height=5.8cm,width=7.2cm,angle=0]{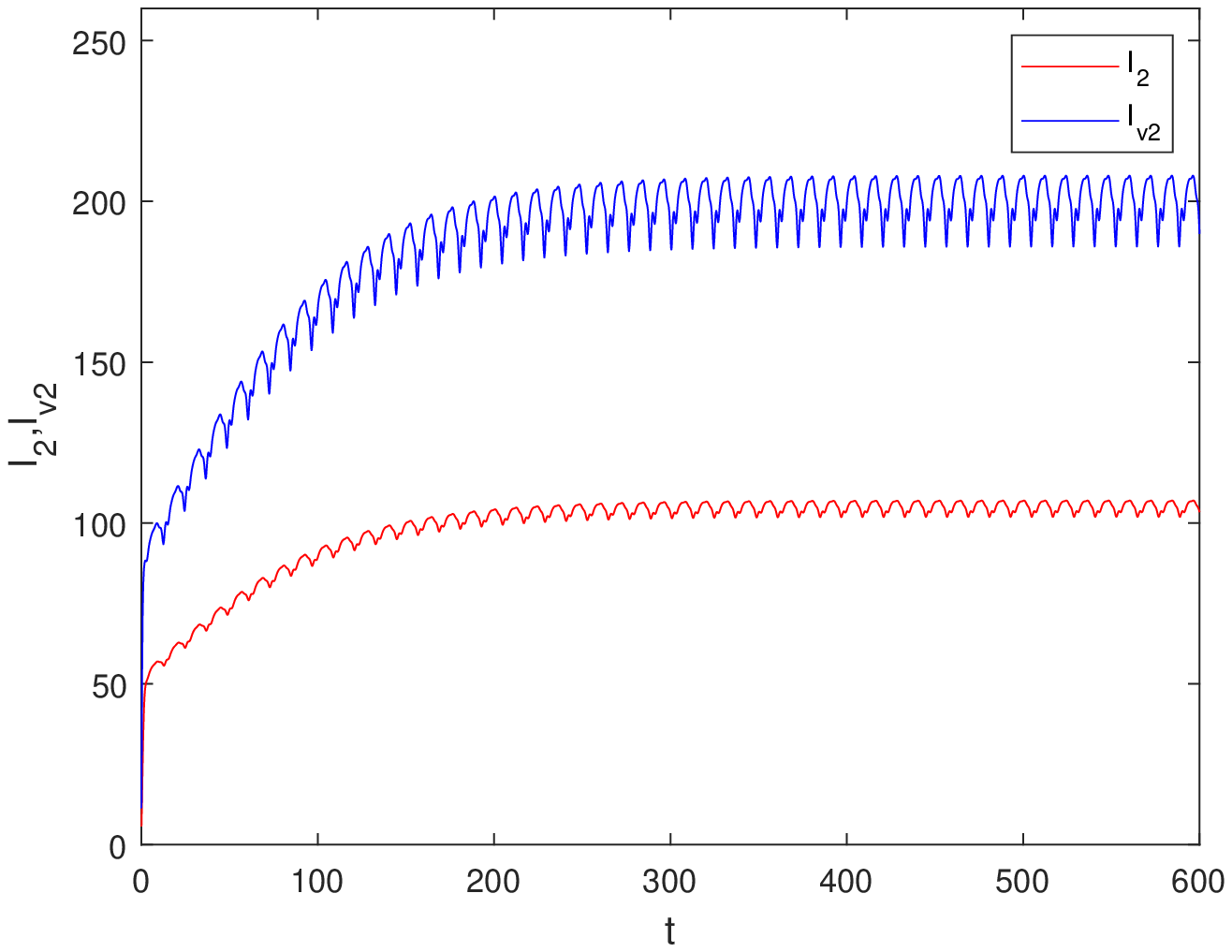}}
\caption{(a) $I_1$ and $I_{v1}$ die out; (b) $I_2$ and $I_{v2}$ persist. The initial data are the same as Fig. \ref{long-term-1}. }\label{long-term-3}
\end{figure}

\subsection{Effects of parameters on $\mathcal{R}_0$}

In order to explore the effect of seasonality, we set the biting rate $\beta(t)\approx a_0(1-b_0\cos(0.523599t))$,  where $a_0$ is the average biting rate, and $b_0\in [0,1]$ is the strength of seasonal forcing. We use the same parameter values as in Case 1 in Section 6.1.
Fig. \ref{seasonality} describes the dependence of $\mathcal{R}_0$ on $a_0$ and $b_0$.
The More precisely,
Fig. \ref{seasonality}(a) shows that $\mathcal{R}_0$ is an increasing function of $a_0$ for fixed $b_0$. Fig. \ref{seasonality}(b) compares the influences of the time-dependent biting rate and the time-averaged biting rate on $\mathcal{R}_0$. As can be seen in Fig. \ref{seasonality}(b), $\mathcal{R}_0$ increases as $b_0$ increases. This implies that the use of the time-averaged biting rate may underestimate the risk of disease transmission. It should be emphasized that this phenomenon is not observed in all malaria models, which is dependent on model parameters.
\begin{figure}[!ht]
\centering
\subfigure[]{
\includegraphics[height=5.8cm,width=7.2cm,angle=0]{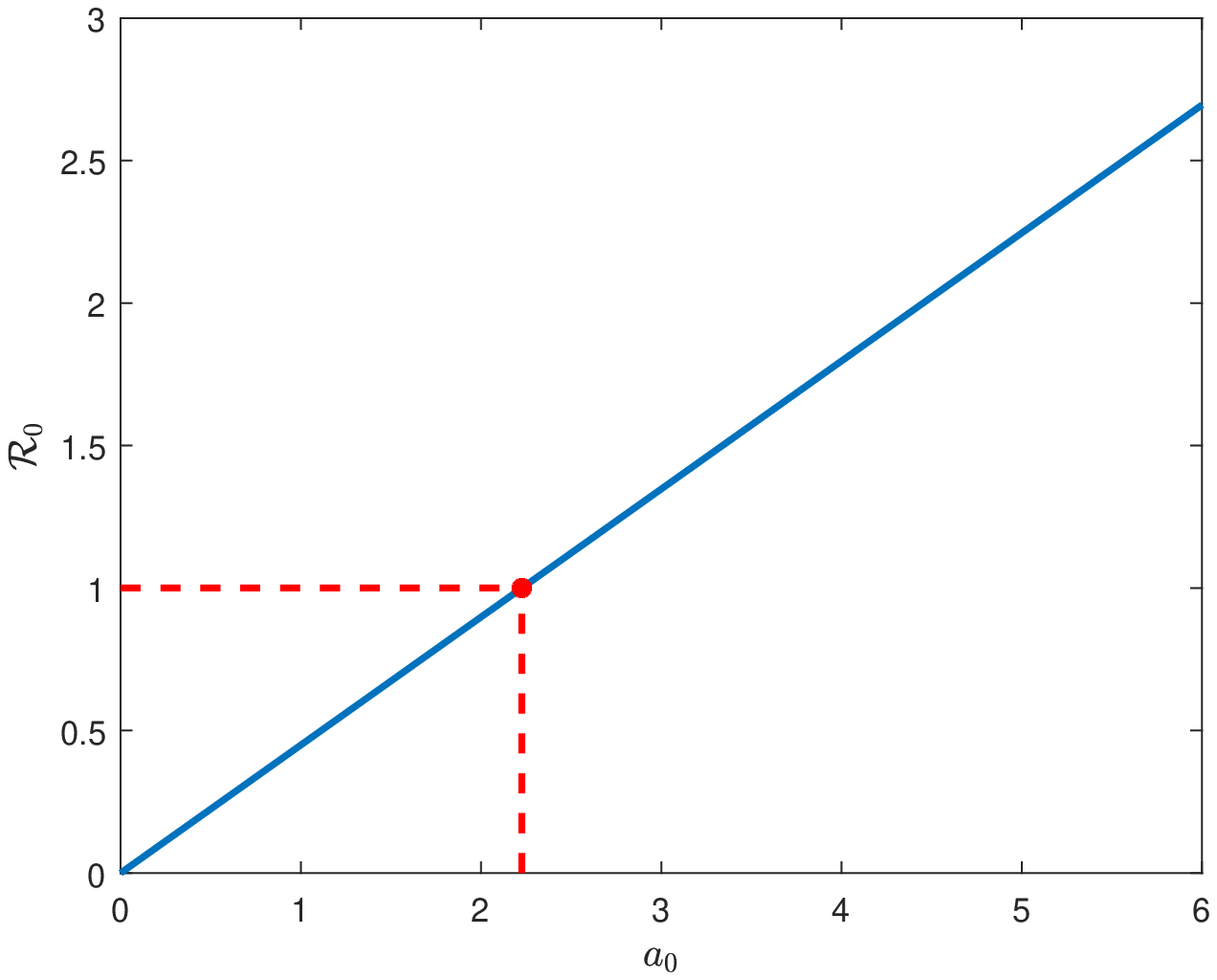}}
\subfigure[]{
\includegraphics[height=5.8cm,width=7.2cm,angle=0]{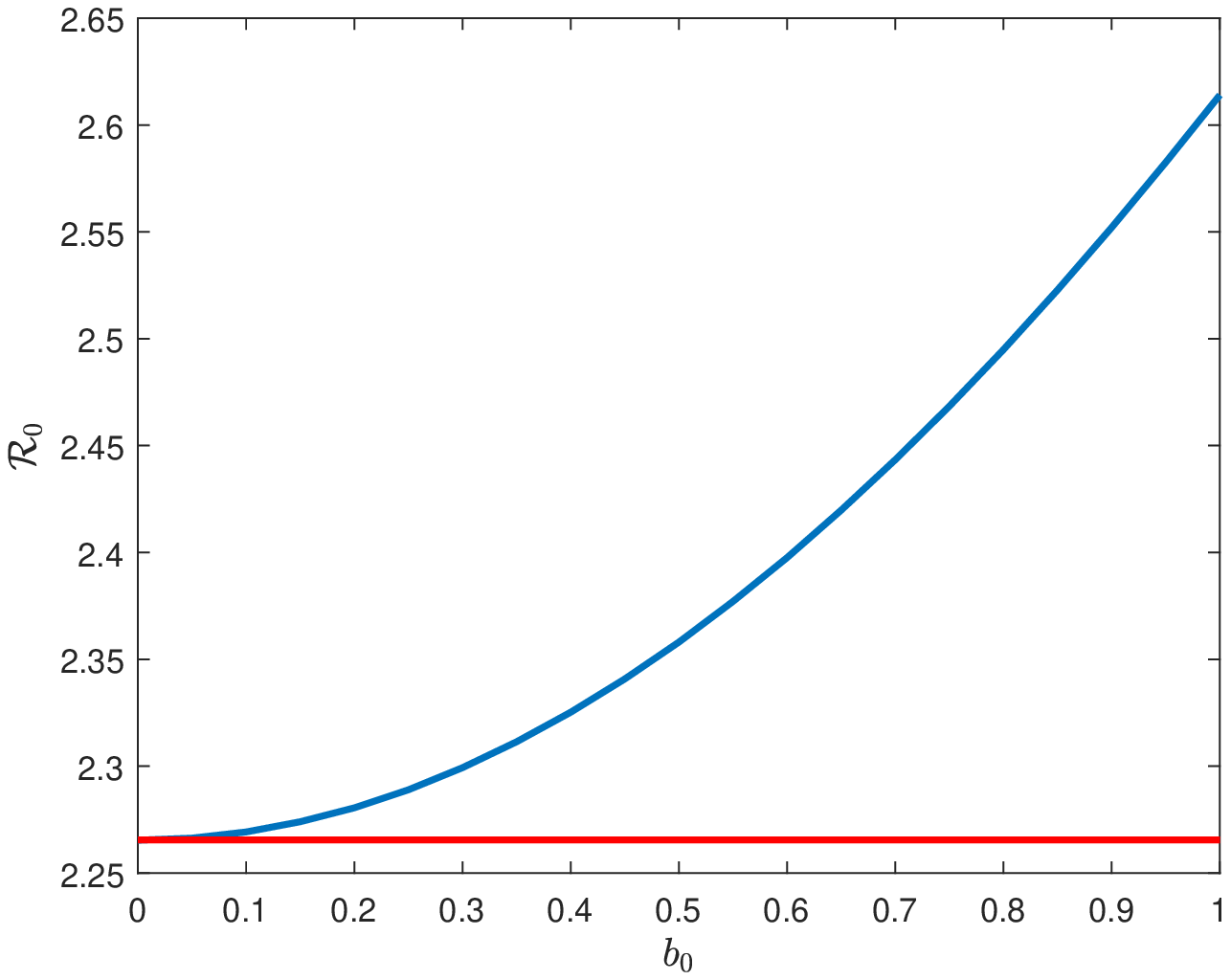}}
\caption{The effect of seasonality on $\mathcal{R}_0$. (a) $\mathcal{R}_0$ as a function of $a_0$ when $b_0=0.35674$; (b) $\mathcal{R}_0$ as a function of $b_0$ when $a_0=5.1492$. }\label{seasonality}
\end{figure}

Next, we investigate the vector-bias effect. We use $q:=l/p$ to measure the relative attractivity of susceptible host versus infection one. Our numerical result in Fig. \ref{q} shows that $\mathcal{R}_0$ decreases as $q$ increases, which indicates that the ignorance of the vector-bias effect will underestimate the value of $\mathcal{R}_0$.
\begin{figure}[!ht]
\centering
\subfigure{
\includegraphics[height=5.8cm,width=7.2cm,angle=0]{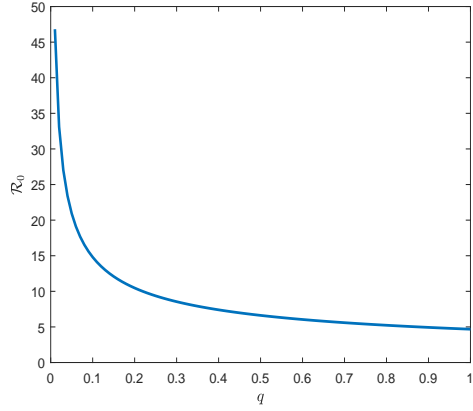}}
\caption{The effect of vector-bias on $\mathcal{R}_0$.  }\label{q}
\end{figure}
In fact, we can analytically prove the monotonicity of $\mathcal{R}_0$ with respect to $q$.
Let $A^i$ and $B^i~(i=1,2)$ be two bounded linear operators on $C_{\omega}(\mathbb{R},\mathbb{E})$ given by
$$[A^iv](t):=\int^{\infty}_{0}\Psi_i(t,t-s)v(t-s)ds,~[B^iv](t):=\mathcal {F}_i(t)v,~\forall t\in\mathbb{R},v\in C_{\omega}(\mathbb{R},\mathbb{E}),$$
where $\Psi_i$ and $\mathcal {F}_i(t)$ are defined as in Section 3. Inspired by Section 4.2 in \cite{Liang et al. 2017}, we write
$$A^iv=(A^i_1v_1, A^i_2v_2),~{\rm and}~B^iv=(B^i_1v_2, B^i_2v_1),~~~\forall v=(v_1,v_2)\in C_{\omega}(\mathbb{R},\mathbb{E}),$$
where
\begin{align*}
&[A^i_1v_1](t)=\int^{\infty}_{0}T_i(t,t-s)v_1(t-s)ds,~~[A^i_2v_2](t)=\int^{\infty}_{0}T_3(t,t-s)v_2(t-s)ds,\\
& [B^i_1v_2](t)=c_i\beta(t,\cdot)v_2(\cdot),~~~[B^i_2v_1](t)=\frac{\alpha_i\beta(t,\cdot)pM^*(\cdot)}{lN(\cdot)}v_1(\cdot),~~~i=1,2.
\end{align*}
According to Section 3.1, $\mathcal{L}_i(q)=A^iB^iv=(A^i_1B^i_1v_2, A^i_2B^i_2v_1)$, it then follows that
$$\mathcal{L}^2_i(q)v=(A^i_1B^i_1A^i_2B^i_2v_1, A^i_2B^i_2A^i_1B^i_1v_2)=\frac{1}{q}\mathcal{L}^2_i(1)v,~~~i=1,2,$$
and hence, $\mathcal{L}^2_i(q)=\frac{1}{q}\mathcal{L}^2_i(1)$. In view of $r^2(\mathcal{L}_i(q))=r(\mathcal{L}^2_i(q))$, we obtain
$$\mathcal{R}_i(q):=r(\mathcal{L}_i(q))=\frac{1}{\sqrt{q}}r(\mathcal{L}_i(1))=\frac{1}{\sqrt{q}}\mathcal{R}_i(1),~~~i=1,2.$$
Therefore, $\mathcal{R}_0(q)=\max\{\mathcal{R}_1(q),\mathcal{R}_2(q)\}=\frac{1}{\sqrt{q}}\max\{\mathcal{R}_1(1),\mathcal{R}_2(1)\}$.
This supports our numerical finding.

\section{Discussion}

In this paper, we have proposed a two-strain malaria model with seasonality and vector-bias. It is of interest to note that our model is a competitive system for sensitive and resistent strains, but the corresponding subsystem of each strain is cooperative. To characterize this mathematical structure, we define a time-dependent region $X(t)$.  Although the introduction of time-varying region brings out some mathematical difficulties, the solution map $Q(t): X(0)\rightarrow X(t)$ is an $\omega$-periodic semiflow. This nice property makes us use uniform persistence theory for model dynamics. Our results show that the zero solution is global attractiveness if $\mathcal {R}_0=\max\{\mathcal {R}_1, \mathcal {R}_2\}<1$ (see Theorem \ref{global attractive}); sensitive (resistent) strains are uniformly persistent if $\mathcal {R}_1>1>\mathcal {R}_2~(\mathcal {R}_2>1>\mathcal {R}_1)$ (see Theorem \ref{Competitive exclusion}); and  the model is uniformly persistent and admits a
positive periodic solution if $\mathcal {R}_1>1, \mathcal {R}_2>1$, $\hat{\mathcal {R}}_1>1$ and $\hat{\mathcal {R}}_2>1$ (see Theorem \ref{uniformly persistent}). We also have
analyzed the asymptotic behavior of the basic reproduction number with small and large diffusion coefficients. Numerically, we have demonstrated the long-time behaviors of solutions: competitive exclusion and coexistence, and revealed the influences of some key parameters on the basic reproduction number. It is found that $\mathcal{R}_0$ increases as
 the strength of seasonal forcing increases, but it is
a decreasing function of the relative attractivity of susceptible host versus infection one.

Finally, we mention that under certain condition, system \eqref{model} is a monotone system with respect to the partial order $\leq_K$, which is induced by the cone $K=\mathbb{E}^+\times(-\mathbb{E}^+)$.
Hence, if we can prove the uniqueness of positive periodic solution in Theorem \ref{uniformly persistent}, then
the positive periodic solution is globally attractive in $X(0)\setminus \{0\}$ by the virtue of
the theory of monotone systems. This is a challenging problem and left for future study.

\end{document}